\documentclass{article}
\usepackage[activeacute,english]{babel} 
\usepackage[utf8]{inputenc}
\usepackage{csquotes}
\usepackage{verbatim} 
\usepackage{graphicx} 
\usepackage[usenames,dvipsnames]{color} 
\usepackage{enumerate,mdwlist}  
\usepackage{multirow} 
\usepackage{amsmath,amsfonts,amssymb,amsthm} 
\usepackage{bm} 
\usepackage{empheq} 
\usepackage{bbm,dsfont} 
\usepackage{mathrsfs} 
\numberwithin{equation}{section} 
\usepackage{float} 
\usepackage[labelfont=bf,labelsep=space]{caption} 
\usepackage{empheq} 

\usepackage[OT2,OT1]{fontenc}
\DeclareRobustCommand\cyr{%
  \renewcommand\rmdefault{wncyr}%
  \renewcommand\sfdefault{wncyss}%
  \renewcommand\encodingdefault{OT2}%
  \normalfont
  \selectfont}
\DeclareTextFontCommand{\textcyr}{\cyr}

\usepackage{subcaption}

\usepackage{geometry}
\usepackage{color}
\definecolor{red}{rgb}{.7,0,0}
\definecolor{blue}{rgb}{0,0,1}

\def\mcB{\mathcal{B}}

\def\mcL{\mathcal{L}}

\def\mcS{\mathcal{S}}

\def\mcO{\mathcal{O}}

\def\mcU{\mathcal{U}}
\def\mcZ{\mathcal{Z}}

\def\bbR{\mathbb{R}}

\def\bbN{\mathbb{N}}
\def\bbI{\mathbb{I}}

\def\bbC{\mathbb{C}}
\def\bbP{\mathbb{P}}
\def\bbE{\mathbb{E}}
\def\bbS{\mathbb{S}}

\def\fkc{\mathfrak{c}}

\def\vol{\mathsf{vol}}

\def\fka{\mathfrak{a}}
\def\fkb{\mathfrak{b}}
\def\fkc{\mathfrak{c}}
\def\fkf{\mathfrak{f}}

\def\fkg{\mathfrak{g}}

\def\fkh{\mathfrak{h}}
\def\fkl{\mathfrak{l}}
\def\fkH{\mathfrak{H}}

\def\fkz{\mathfrak{z}}

\DeclareMathOperator{\conv}{conv}

\DeclareMathOperator{\rank}{rank}

\def\diff{\mathrm{D}}

\DeclareSymbolFont{extraitalic}      {U}{zavm}{m}{it}
\DeclareMathSymbol{\Qoppa}{\mathord}{extraitalic}{161}
\DeclareMathSymbol{\qoppa}{\mathord}{extraitalic}{162}

\DeclareMathOperator{\vertex}{V}
\def\var{\mathtt{v}}
\def\liftfunct{\textrm{\cyr l}} 

\usepackage[pdfencoding=auto,unicode]{hyperref} 
\hypersetup{
 pdftoolbar=true,
 pdfmenubar=true,
 pdfstartview={FitV},
 pdftitle={How `few' zeroes of a random fewnomial system are real?},
 pdfauthor={Alperen A. Ergür, Máté L. Telek, Josu\'{e} Tonelli-Cueto},
 pdfsubject={probability, real, algebraic geometry},
 pdfkeywords={real algebraic geometry,polynomial systems},
 pdfcreator={overleaf}, 
 pdfproducer={overleaf},
 linkcolor=red     
 citecolor=green   
 urlcolor=cyan     
 filecolor=magenta 
}

\usepackage{bookmark}

\title{On the Number of Real Zeros\\of Random Sparse Polynomial Systems}
\author{
Alperen A. Ergür\\
\footnotesize{University of Texas at San Antonio}\\
\footnotesize{Dept. of Mathematics}\\
\footnotesize{San Antonio, Texas, USA}\\
\footnotesize{\tt  alperen.ergur@utsa.edu}
\and
Máté L. Telek\\
\footnotesize{Budapest University of Technology and Economics}\\
\footnotesize{Institute of Mathematics}\\
\footnotesize{Budapest, HUNGARY}\\
\footnotesize{\tt mtelek@math.bme.hu}
\and
Josu\'{e} Tonelli-Cueto\\
\footnotesize{CUNEF Universidad}\\
\footnotesize{Dept. of Quantitative Methods}\\
\footnotesize{Madrid, SPAIN}\\
\footnotesize{\tt  josue.tonelli.cueto@bizkaia.eu} 
}
\date{}

\makeatletter
\def\th@plain{%
  \thm@notefont{}
  \slshape 
}
\def\th@definition{%
  \thm@notefont{}
  \normalfont 
}
\makeatother
\theoremstyle{plain}
\newtheorem{lem}{Lemma}[section]
\newtheorem{prop}[lem]{Proposition}
\newtheorem{theo}[lem]{Theorem}
\newtheorem{cor}[lem]{Corollary}

\theoremstyle{definition}
\newtheorem{defi}[lem]{Definition}
\theoremstyle{remark}

\newtheorem{remark}[lem]{Remark}

\usepackage{multicol}
\usepackage[linesnumbered,ruled,algosection,vlined]{algorithm2e}
\SetKwInOut{postcondition}{Postcondition}
\SetKwInOut{precondition}{Precondition}
\SetKwInOut{correctness}{Successful if}

\SetCommentSty{mycommfont}
\let\oldnl\nl
\newcommand{\nonl}{\renewcommand{\nl}{\let\nl\oldnl}}

\makeatletter
\let\original@algocf@latexcaption\algocf@latexcaption
\long\def\algocf@latexcaption#1[#2]{%
  \@ifundefined{NR@gettitle}{%
    \def\@currentlabelname{#2}%
  }{%
    \NR@gettitle{#2}%
  }%
  \original@algocf@latexcaption{#1}[{#2}]%
}
\makeatother

\usepackage[
backend=biber,
style=numeric,
sorting=nty,
maxbibnames=99
]{biblatex}
\AtEveryBibitem{\clearfield{issn}}
\AtEveryBibitem{\clearfield{isbn}}
\AtEveryBibitem{\clearfield{url}}
\AtEveryBibitem{\ifentrytype{book}{\clearfield{pages}}{}}

\begin{document}
\maketitle

\begin{abstract}
Consider a random system $\mathfrak{f}_1(x)=0,\ldots,\mathfrak{f}_n(x)=0$ of $n$ random real polynomials in $n$ variables, where each $\mathfrak{f}_k$ has a prescribed set of exponent vectors in a set $A_k\subseteq \mathbb{Z}^n$ of size $t_k$. Assuming that the coefficients of the $\mathfrak{f}_k$ are independent Gaussian of any variance, we prove that the expected number of zeros of the random system in the positive orthant is bounded from above by $4^{-n} \prod_{k=1}^n t_k(t_k-1)$. This result is a probabilistic version of Kushnirenko's conjecture; it provides a bound that only depends on the number of terms and is independent of their degree.
\end{abstract}

\noindent \textbf{Mathematics Subject Classification (MSC) 2020}
\\14P05, 60D05
\section{Introduction}
Instances of polynomial system solving arise in a variety of problems coming from kinematics~\cite{sommesewampler}, dynamical systems \cite{fantuzzi-dynamicalsystems},  mathematical modeling of chemical reaction networks~\cite{dickenstein2020algebraic}, computer vision \cite{kukelova-minimalvision}, and computer aided geometric design \cite{cox2020applications,sederberg1986algebraic}. The common thread among these plethora of applications is that the polynomial systems that occur are structured, and the solutions that are most useful are the real ones. In this setting, the basic question is: how many real zeros are there?

When we consider zeros over complex numbers, we have the power of intersection theory~\cite{eisenbudharris2016}. In particular, for sparse polynomial systems the number of complex solutions is given by the celebrated BKK bound~\cite{bernshtein1975,kushnirenko1976french} (cf. \cite[Chapter 3]{sottile2011book}). The story becomes more complicated and interesting for real zeros. Consider the following simple polynomial system:
\[
\left.
\begin{array}{ll}
a_1+b_1X+\gamma_1Y+c_1XYZ^d=0\\
a_2+b_2X+\gamma_2Y+c_2XYZ^d=0\\
a_3+b_3X+\gamma_3Y+c_3XYZ^d=0.
\end{array}
\right.
\]
where $d\in\bbN$ and the $a_k$, $b_k$, $c_k$ and $d_k$ are real numbers. Although the system has, generically, $d$ complex zeros, it has at most two real solutions. This example illustrates that the number of real zeros cannot be estimated using the number of complex solutions.

The above phenomenon is a general one. Khovanski\u{\i}, in his seminal book \emph{Fewnomials}~\cite{khovanskii1991book}, showed that the maximum number of real zeros of a polynomial system can be bounded in terms of the description complexity of the system\footnote{Although it should be emphasized that Khovanski\u{\i}'s work goes well beyond the polynomial case up to the Pfaffian case.}. The term \emph{fewnomial} refers to a polynomial with few monomials, and so to a polynomial with a low description complexity. He showed that independent of the degree of its monomials, real fewnomial systems have few real zeros. 

The work of Khovanski\u{\i} was an answer to several conjectures by Kushnirenko in 1977~\cite{kushnirenkoletter}:
A regular positive zero of a polynomial system $f=(f_1,f_2,\ldots,f_n)$ is a point $x \in \mathbb{R}^n$ with positive coordinates ($x_i>0$) such that the Jacobian matrix of the polynomial system $f$ at $x$, $\diff_xf$, is full-rank. We denote the set of regular positive real zeros by  $\mcZ_r(f,\bbR^n_+)$. Among the conjectures of Kushnirenko, the following one remains widely open:
\begin{quote}
\emph{Kushnirenko's Question}: Let $A_1,\ldots,A_n\subset \bbN^n$ be finite sets of sizes $t_1,\ldots,t_n$ and $f$ the system of fewnomials given by
\begin{equation*}
\left.
    \begin{array}{rl}
        f_1&=\sum_{\alpha\in A_1}f_{1,\alpha}X^\alpha\\
        &\,\vdots\\
        f_n&=\sum_{\alpha\in A_n}f_{n,\alpha}X^\alpha\\
    \end{array}
\right.
\end{equation*}
Is there a bound of the form
\[
\#\mcZ_r(f,\bbR^n_+)\leq \mathrm{poly}(t_1,\ldots,t_n)^n?
\]
\end{quote}

It is fair to say that Kushnirenko's question is one of the most challenging problems in real algebraic geometry. However, its reach and importance goes far beyond. In complexity theory, variants of Kushnirenko's question (where the bound should be explicit in an algebraic complexity measure), such as the real $\tau$-conjecture~\cite{koiran2010} or the adelic $\tau$-conjecture~\cite{phillipsonrojas2014}, imply Valiant's algebraic variant of $\tt{P}$ vs $\tt{NP}$ (see~\cite{burgisser2000VPvsVNPbook} for details and relation to classical questions in complexity theory, and \cite{briquelburgisser2020} for a probabilistic take on the problem). 

Specific cases of Kushnirenko's question play an important role in studying chemical reaction networks. In that setting, the steady states of the network correspond to positive real zeros of a certain parameterized polynomial system~\cite{dickenstein2020algebraic}. Although there exist several methods to decide whether the system has at least two positive zeros~\cite{bihandickensteingiaroli2020,muellerfeliuregensburgerconradishiu2013,millendickensteinshiuconradi2011,feliuwiuf2013,feliusadehimanesh2022}, finding (tight) upper bounds is a hard and open problem. The polynomial system arising from a reaction network is usually sparse, and the maximum number of its positive zeros is much smaller compared to the number of complex zeros~\cite{grossharringtonrosensturmfels2015,obatakeshiutangtorres2019}. Thus techniques from real algebraic geometry should yield insight into the number of steady states~\cite{feliuhelmer2018,giarolirischtermillandickenstein2019,bihandickensteingiaroli2020}. 

Khovanski\u{\i} \cite[\S3.14, Corollary 4]{khovanskii1991book} obtained a bound of the form
\[
2^{\binom{t-1}{2}}(n+1)^{t-1}
\]
where $t=\#\left(\cup_{k=1}^n A_k\right)$ is the number of types of monomials appearing in the system. Later, Bihan and Sottile \cite{bihansottile2007,bihansottile2011} improved this bound to
\[
\frac{\mathrm{e}^2+3}{4}\, 2^{\binom{t-n-1}{2}} n^{t-n-1}
\]
and, in a special mixed case in which $t=\sum_{k=1}^nt_k-n+1$, further improved this bound to
\[
\frac{\mathrm{e}^2+3}{4}\, 2^{\binom{t-n-1}{2}} \binom{t-n-1}{t_1-2,\ldots,t_n-2}.
\]
After four decades of hard work, as of today, we are still far from answering Kushnirenko's question.  We only have precise bounds for special cases \cite{bihanelhilany2017,koiranportiertavenas2015,lirojaswang2003}, and even the bivariate case of Kushnirenko's question remains open~\cite{koiranportiertavenas2015}. We refer the reader to \cite{sottile2011book} for a comprehensive exposition, to~\cite{dickentein2020survey} for a short survey and to \cite[Appendix F, \S1]{tonellicuetothesis} for a historical survey. 

Our current inability to answer Kushnirenko's question, and the pressing need for an answer in applications, motivates a probabilistic take on the problem. How many zeros of a random real fewnomial system are real?  This question turns out the be harder than it sounds. There is a well-developed theory (either relying on the kinematic formula~\cite{azaiswschebor2009,edelemankostlan1995} or convex bodies~\cite{breidingburgisserlerariomathis2022,kazarnovskii2020,mathisstecconi2022}) in random real algebraic geometry. This theory has produced  strong results on the expected number of real zeros of random dense polynomials, that is, random polynomials supported on all monomials of a given degree \cite{kostlan1993,shubsmale1993,armentanoazaisdalmaoleon2018,letendre2019,letendrepuchol2019,armentanoazaisdalmaoleon2021,armentanoazaisdalmaoleon2022}. 
However, in the case of random structured polynomials, applying the theory becomes more complicated \cite{breidingfairchildsanrsieroshehu2022,kazarnovskii2022}, and earlier available results are either for random polynomials supported on a product of simplices \cite{rojas1995} or depend strongly on the degree of the monomials \cite{malajovichrojas2002}\footnote{Recently, Malajovich~\cite{malajovich2013,malajovich2023} has provided a bound that bridges between the dense and the sparse cases. In a different direction, there are results on complex zeros of random sparse complex polynomials ~\cite{bayraktarkicisel2023,shiffmanzeldtich2004,shiffmanzeldtich2010}.}. In summary, the aforementioned results in random real algebraic geometry do not address Kushnirenko's question. In~\cite{BETC-fewnomials}, we obtained a bound for random fewnomial systems with fixed support ($A=A_1=\cdots=A_n$) and centered Gaussian random coefficients with the same variance structure. This was relevant to Kushnirenko's question because the bound depended only on the number of distinct exponent vectors and was independent of the underlying variances (within the probabilistic model we restrict to). We emphasize the variance independence because almost all prior results placed strong assumptions on the variance structure.

Recently, Bürgisser~\cite{burgisser2023} generalized the results of~\cite{BETC-fewnomials} to the mixed case (where $A_1,A_2,\ldots,A_n$ are not necessarily equal). However, the result in \cite{burgisser2023} only holds for centered Gaussian fewnomial systems in which all variances are equal to one. In this paper, we obtain estimates that hold for arbitrary support sets and arbitrary variances in the same spirit with Kusnirenko's question.

The simplest form of our result can be seen in the theorem below. In the remainder of the introduction we explain the consequences of our result in specific cases and state our main theorem in full technical detail. We also provide side results on the volume of random fewnomial varieties.  
\begin{theo} \label{theo:MainTheoremNumberZeros}
Let $A_1,\ldots,A_n\subset \bbR^n$ be finite sets of sizes $t_1,\ldots,t_n$, and $\fkf$ the system of $n$ random fewnomials given by
\begin{equation}
\left.
    \begin{array}{rl}
        \fkf_1&=\sum_{\alpha\in A_1}\fkf_{1,\alpha}X^\alpha\\
        &\,\vdots\\
        \fkf_n&=\sum_{\alpha\in A_n}\fkf_{n,\alpha}X^\alpha\\
    \end{array}
\right.
\end{equation}
where the $\fkf_{k,\alpha}$ are independent centered Gaussian random variables (with no restrictions placed on variances). Then we have
\begin{equation}
    \bbE_{\fkf} \#\mcZ(\fkf,\bbR^n_+)\leq \frac{1}{4^n} \prod_{k=1}^n t_k(t_k-1)
\end{equation}
where $\mcZ(\fkf,\bbR^n_+)$ is the zero set of $\fkf$ in $\bbR_+^n$.
\end{theo}

\begin{remark}
Observe that if the affine span of $\bigcup_{k=1}^nA_k$ is not all of $\bbR^n$ (or, more generally, if $A_1,\ldots,A_n$ does not contain an independent transversal (see \cite[Lemma 1]{yu2016} for the precise definition)), then $\mcZ(\fkf,\bbR^n_+)=\varnothing$ with probability one. For instance, if the mixed volume of the convex hulls of the $A_i$ is zero, then $\mcZ(\fkf,\bbR^n_+)=\varnothing$ with probability one.
\end{remark}

\subsection{Random Mixed Fewnomials with Restricted Variances}

The following theorem shows that the bound of Theorem~\ref{theo:MainTheoremNumberZeros} can be improved when we impose restrictions on the variances on the random fewnomial system. This result improves the main theorem of B\"urgisser~\cite[Theorem~1.1]{burgisser2023} by allowing more flexible assumptions on the variance structure and a better multiplying factor: $4^{-n}$ instead of $(2\pi)^{-n/2}$. We give the proof in Section~\ref{sec:proofMainResult}.

\begin{theo}\label{theo:MainTheoremNumberZeros_variances}
Under the same notations and assumptions of Theorem~\ref{theo:MainTheoremNumberZeros}, assume instead that the variance vectors 
\[
\var_k=(\var_{k,\alpha})_{\alpha\in A_k}
\]
of the $\fkf_k$ satisfy the following conditions:
\begin{enumerate}
    \item[(VM1)] for all $k$ and all $\alpha\in A_k$, $\var_{k,\alpha} \leq 1$.
    \item[(VM2)] for all $k$ and $\alpha\in A_k$ a vertex of the polytope $P_k:=\conv(A_k)$, $\var_{k,\alpha} = 1$.
\end{enumerate} 
Then
\begin{equation}
\bbE_{\fkf}\#\mcZ_r(\fkf,\bbR^n_+)\leq \frac{1}{4^n}\vertex\left(\sum_{k=1}^nP_k\right)\prod_{k=1}^n(t_k-1)
\end{equation}
where $\vertex\left(\sum_{k=1}^nP_k\right)$ is the number of vertices of the Minkowski sum  $\sum_{k=1}^nP_k$. 
\end{theo}
\begin{remark}
Interestingly, in the univariate case ($n=1$), the obtained upper bound, $(t_1-1)/2$, agrees with the average value of the bound given by Descartes' rule of signs~\cite{descartesgeometry}: it is the average number of sign variations assuming coefficients are positive and negative with probability one half. However, optimal results in the univariate case, when all variances are equal to one, were already obtained in~\cite{jindalpandeyshuklazisopoulos2020}.
\end{remark}
 
\subsection{Random Unmixed Fewnomials with Restricted Variances}

In the case where the system is unmixed and  some restrictions are imposed on the variances we obtain better bounds. This result improves the main theorem of \cite[Theorem~1.2]{BETC-fewnomials} (see Remark \ref{rem:unmixedimproved} for details). We give the proof of this in Section~\ref{sec:proofMainResultUnmixed}.

\begin{theo}\label{theo:MainTheoremNumberZeros_variancesunmixed}
Under the same notations and assumptions of Theorem~\ref{theo:MainTheoremNumberZeros}, assume that
\[
A=A_1=\cdots=A_n
\]
for some $A\subset\bbR^n$ of size $t$, and assume that the variance vectors $\var_k$ of the $\fkf_k$ either satisfy both the assumptions (VM1) and (VM2) in Theorem \ref{theo:MainTheoremNumberZeros_variances} or they satisfy
\begin{equation}\label{eq:varindpol}
    \var_1=\cdots=\var_n.
\end{equation}
Then 
\begin{equation}
    \bbE_{\fkf}\#\mcZ(\fkf,\bbR^n_+)\leq \frac{n+1}{4^n}\binom{t}{n+1}.
\end{equation}
\end{theo}
\begin{remark}\label{rem:unmixedimproved}
Theorem \ref{theo:MainTheoremNumberZeros_variancesunmixed} improves \cite[Theorem~1.2]{BETC-fewnomials} in the following sense: The main result of \cite[Theorem~1.2]{BETC-fewnomials} only holds under the assumption \eqref{eq:varindpol} on the variance vectors, whereas Theorem \ref{theo:MainTheoremNumberZeros_variancesunmixed} has more flexible assumptions. Moreover, the upper bound in \cite[Theorem~1.2]{BETC-fewnomials} is of the form $\frac{1}{2^{n-1}} \binom{t}{n}$. Therefore, Theorem \ref{theo:MainTheoremNumberZeros_variancesunmixed} gives a smaller bound if and only if  $t-n \leq 2^{2n-n+1}$.     

\end{remark}

A specific case that attracted considerable attention is when $\# A=n+\ell$ and $\ell$ is a fixed constant~\cite{bihansottile2007}, specially when $\ell=2$ \cite{bihan2007,bihandisckenstein2017,rojas2022counting}. The following corollary (whose proof is at the end of Section~\ref{sec:proofMainResultUnmixed}) shows that if such systems are generated randomly then it is hard to find any positive zero at all. 

\begin{cor}\label{cor:circuitcase}
Under the same the same notations and assumption of Theorem~\ref{theo:MainTheoremNumberZeros_variancesunmixed}, assume that for some constant $\ell\in\bbN$,
\[
\#A=n+\ell
\]
Then
\begin{equation}
\bbP_{\fkf}\left(\mcZ(\fkf,\bbR^n_+)\neq \varnothing \right)\leq \frac{\ell(n+1)^\ell}{4^n}.  
\end{equation}
Moreover, since $\ell$ is a fixed constant, we have that $\lim_{n\to \infty} \bbP_{\fkf}\left(\mcZ(\fkf,\bbR^n_+)\neq \varnothing \right)=0$. 
\end{cor}
\begin{remark}
If we substitute the constant $\ell$ by logpolynomial function $\ell(n)=\mcO(\log^a(n))$ with $a>0$ a constant, the limit probability would still be zero.  
\end{remark}

\subsection{Main Result in Full Technical Detail}

Recall that given a finite set $A\subset\bbR^n$ and a map
\[
\liftfunct:A\rightarrow \bbR,
\]
the \emph{upper envelope of $A$ with respect $\liftfunct$} is the convex polyhedron given by

\begin{equation}
\mcL(A,\liftfunct):=\conv\left\{\begin{pmatrix} \liftfunct(\alpha)-s\\\alpha\end{pmatrix}\mid \alpha\in A,\,s\geq 0\right\}.
\end{equation}
Intuitively, this set arises when we lift $A$ according to $\liftfunct$ and we look at the convex hull of these lifted points,
\[
P^\liftfunct:=\conv\left\{\begin{pmatrix} \liftfunct(\alpha)\\\alpha\end{pmatrix}\mid \alpha\in A\right\},
\]
from above---see Figure~\ref{fig:exam1} for an example. We can now state our main theorem, from which  Theorems~\ref{theo:MainTheoremNumberZeros} and \ref{theo:MainTheoremNumberZeros_variances} will also follow. 

\begin{theo}\label{theo:maintheorem_fulldetails}
Let $A_1,\ldots,A_n\subset \bbR^n$ be finite sets of sizes $t_1,\ldots,t_n$, and $\fkf$ the the system of $n$ random fewnomials given by
\begin{equation*}
\left.
    \begin{array}{rl}
        \fkf_1&=\sum_{\alpha\in A_1}\fkf_{1,\alpha}X^\alpha\\
        &\,\vdots\\
        \fkf_n&=\sum_{\alpha\in A_n}\fkf_{n,\alpha}X^\alpha\\
    \end{array}
\right.
\end{equation*}
where the $\fkf_{k,\alpha}$ are independent centered Gaussian random variables. Consider, for each $\fkf_k$, its variance vector
\[
\var_k=(\var_{k,\alpha})_{\alpha\in A_k}
\]
and construct the lifting functions $\liftfunct_k:A_k\rightarrow \bbR$ given by
\begin{equation}\label{eq:liftingfunction}
    \liftfunct_k:\alpha\mapsto \frac{1}{2}\ln \var_{k,\alpha}.
\end{equation}
Then we have
\begin{equation}\label{eq:mainexpectedzerosbound}
    \bbE_{\fkf}\#\mcZ_r(\fkf,\bbR^n_+)\leq \frac{1}{4^n}\vertex\left(\sum_{k=1}^n\mcL(A_k,\liftfunct_k)\right) \prod_{k=1}^n(t_k-1)
\end{equation}
where $\vertex\left(\sum_{k=1}^n\mcL(A_k,\liftfunct_k)\right)$ is the number of vertices of the Minkowski sum $\sum_{k=1}^n\mcL(A_k,\liftfunct_k)$.
\end{theo}

We note that one can interpret the expression
\[
\vertex\left(\sum_{k=1}^n\mcL(A_k,\liftfunct_k)\right)
\]
as the number of elements of $\sum_{k=1}^nA_k$ that are used in the regular mixed subdivision induced by the lifting functions $\liftfunct_1,\ldots,\liftfunct_n$ which are defined by the logarithm of the variances. Indeed, we find the theory of regular (mixed) subdivisions \cite{deloerarambausantosbooktriangulations} at the core of our development---see Figure~\ref{fig:exam1}. Subdivisions have been used earlier in the work of Bihan, Santos, and Spaenlehauer \cite{bihansantosspaenlehauer2018} and Sturmfels~\cite{sturmfels1994zeros} for constructing fewnomial systems with many zeros, in the work of Sturmfels on the Newton polytope of $A$-resultants \cite[Sect. 2]{sturmfels1994}, in the seminal work by Gelfand, Kapranov and Zelevinsky \cite{gkzbook}, and implicitly in Viro's patchworking method  \cite{viro2008sixteenth} (see \cite[Section 2.2]{ergur2022polyhedral} for an exposition). Further, we can find a lifting `similar' to ours in the work of Avendaño, Kogan, Nisse and Rojas~\cite{avendanokogannisserojas2018} regarding complex fewnomials. We don't have any good explanation for why mixed regular subdivisions appear in random fewnomial theory, but there seems to be a deep connection to be uncovered.

\begin{figure}[hbt]
    \centering
    \begin{subfigure}[t]{0.33\textwidth}
        \centering
        \includegraphics[width=0.6\textwidth]{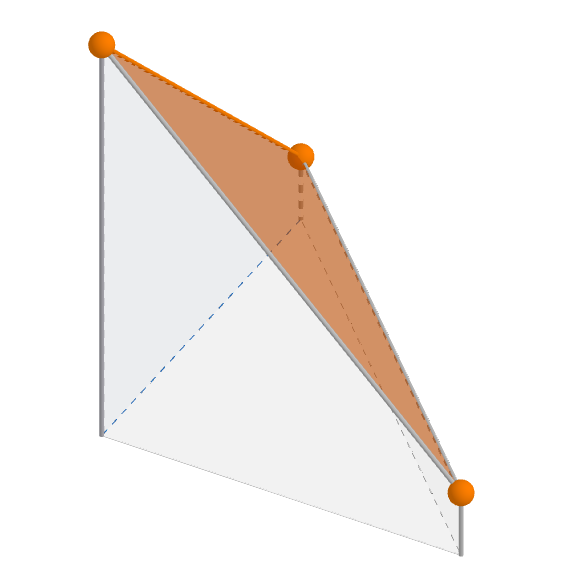}
        \caption{$\mcL(A_1,\liftfunct_1)$}
    \end{subfigure}\hfil
    \begin{subfigure}[t]{0.33\textwidth}
        \centering
        \includegraphics[width=0.6\textwidth]{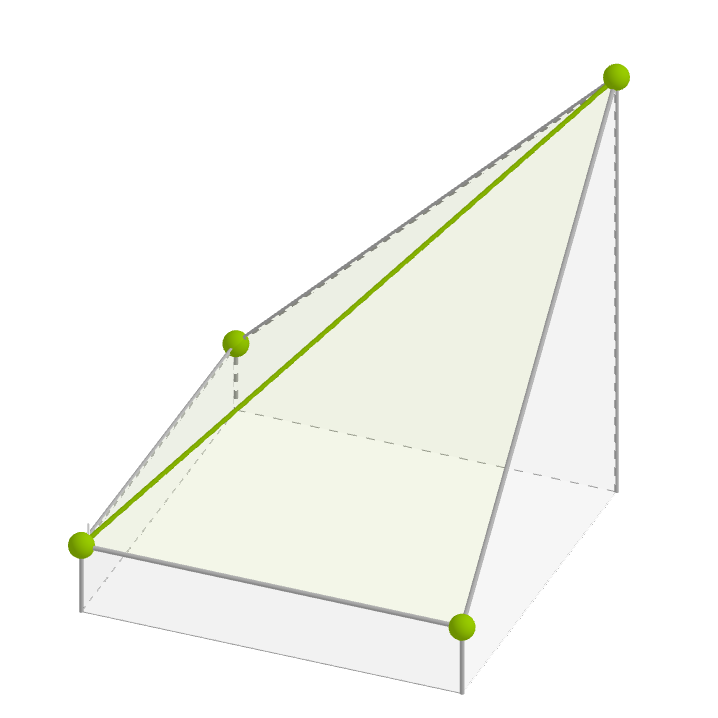}
        \caption{$\mcL(A_2,\liftfunct_2)$}
    \end{subfigure}\hfil
    \begin{subfigure}[t]{0.33\textwidth}
        \centering
        \includegraphics[width=0.6\textwidth]{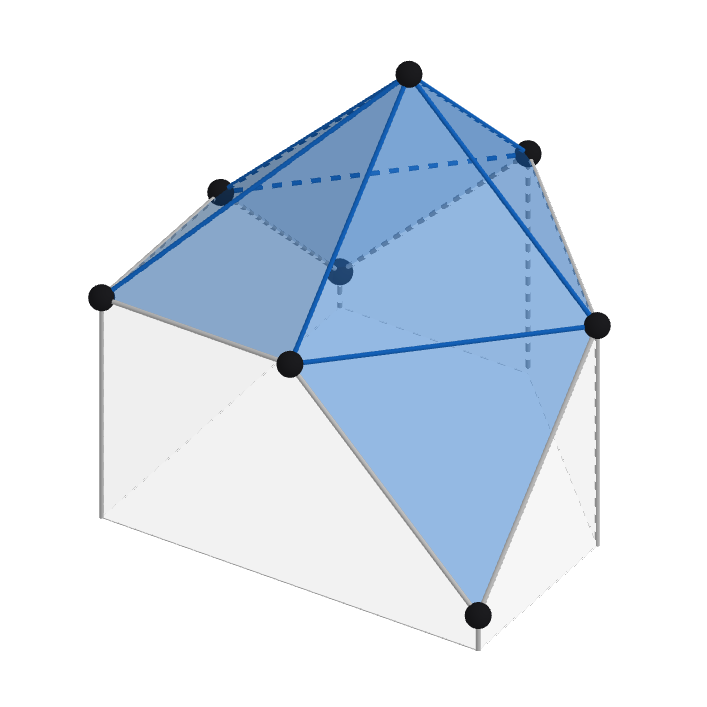}
        \caption{$\mcL(A_1,\liftfunct_1)+\mcL(A_2,\liftfunct_2)$}
    \end{subfigure}
    \medskip
        \begin{subfigure}[b]{0.33\textwidth}
        \centering
        \includegraphics[width=0.6\textwidth]{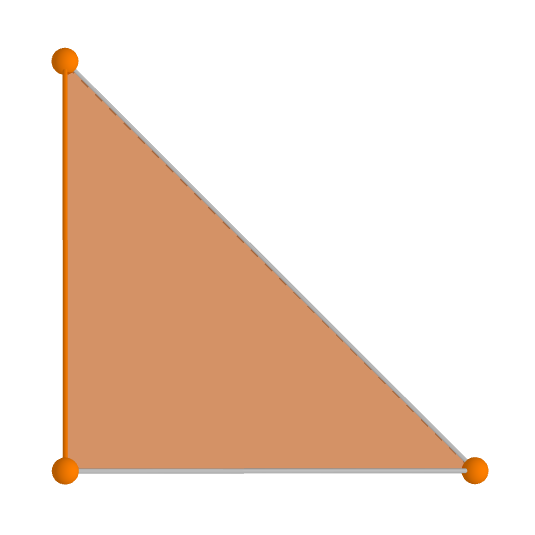}
        \caption{Regular subdivision on $A_1$\\induced by $\liftfunct_1$ }
    \end{subfigure}\hfil
    \begin{subfigure}[b]{0.33\textwidth}
        \centering
        \includegraphics[width=0.6\textwidth]{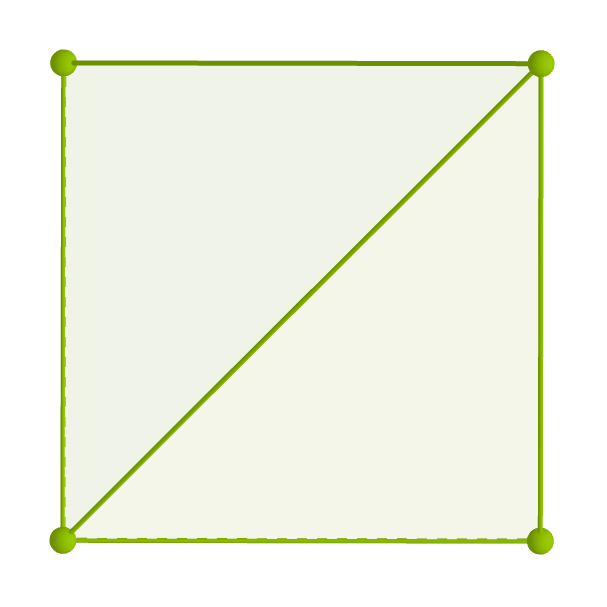}
        \caption{Regular subdivision on $A_2$\\induced by $\liftfunct_2$}
    \end{subfigure}\hfil
    \begin{subfigure}[b]{0.33\textwidth}
        \centering
        \includegraphics[width=0.6\textwidth]{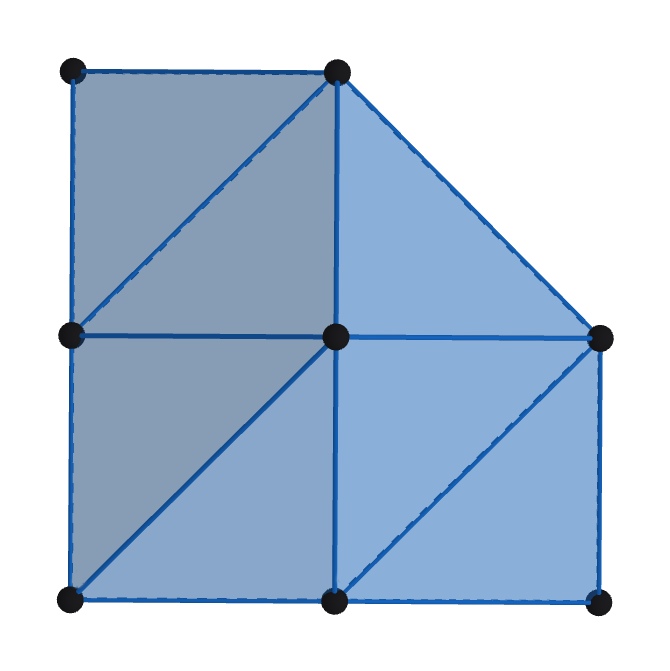}
        \caption{Mixed regular subdivision\\on $A_1+A_2$ induced by $(\liftfunct_1,\liftfunct_2)$ }
    \end{subfigure}
    \caption{Example of $\mcL(A_1,\liftfunct_1)+\mcL(A_2,\liftfunct_2)$ 
    (in the first row) and the induced mixed regular subdivision (in the second row) with $A_1=\{(0,0),(1,0),(0,0)\}$ with $\liftfunct_1((0,0))=1$, $\liftfunct_1((1,0))=0$, $\liftfunct_1((0,1))=0$ and $A_1=\{(0,0),(1,0),(0,1),(1,1)\}$ with $\liftfunct_2((0,0))=0$, $\liftfunct_2((1,0))=0$, $\liftfunct_2((0,1))=0$, $\liftfunct_2((1,1))=1$.}
    \label{fig:exam1}
\end{figure}

\subsection{Bounds on Volume of Random Projective Fewnomial Varieties}

In complex algebraic geometry, the (normalized) volume of a complex algebraic variety is just the degree \cite[Corollary~20.10]{conditionbook}. In real algebraic geometry, we can only talk about the average volume of a random variety--- (normalized) volume is also referred to as  ``average degree'' ~\cite{burgisserlerario2020}. Shub and Smale~\cite{shubsmale1993} showed that for a random KSS (Kostlan-Shub-Smale) homogeneous polynomial system $\fkf$ of degrees $d_1,\ldots,d_q$ in the $n+1$ variables $X_0,X_1,\ldots,X_n$, we have that
\begin{equation}\label{eq:kssvolume}
    \bbE_{\fkf}\,\vol_{n-q}\mcZ(\fkf,\bbP^n_{\bbR})=\frac{\vol_{n-q}\bbP^{n-q}_{\bbR}}{\vol_{n}\bbP^{n}_{\bbR}}\,\sqrt{\prod_{k=1}^q d_k},
\end{equation}
where $\vol_{k}\bbP^{k}_{\bbR}$ is the volume under the usual Riemannian structure. 
Moreover, we even have a central limit theorem in the case that all degrees are equal~\cite{armentanoazaisdalmaoleon2022}. The following result shows that, for random projective varieties given by random fewnomials, the average volume can be bounded independently of the degree. We give the proof in Section~\ref{sec:volumes}.

\begin{theo}\label{theo:MainTheoremNumberZeros_volume}
Let $q\leq n$, $d_1,\ldots,d_q\in \bbN$ and $A_1,\ldots,A_q\subset \bbN^{n+1}$ be finite subsets of sizes $t_1,\ldots,t_q$ such that for all $k$,
\[\{d_k e_0,\ldots,d_k e_n\}\subseteq A_k\subseteq d_k\Delta_n\]
where $\{e_0,\ldots,e_n\}$ is the standard basis of $\bbR^{n+1}$ and $\Delta_n:=\conv(e_0,\ldots,e_n)$ the $n$-simplex. Consider the random homogeneous fewnomial system $\fkf$ given by
\begin{equation*}
\left.
    \begin{array}{rl}
        \fkf_1&=\sum_{\alpha\in A_1}\fkf_{1,\alpha}X^\alpha\\
        &\,\vdots\\
        \fkf_q&=\sum_{\alpha\in A_q}\fkf_{q,\alpha}X^\alpha\\
    \end{array}
\right.
\end{equation*}
where the $\fkf_{k,\alpha}$ are independent centered Gaussian random variables (with no restrictions placed on the variances). Then, with probability one, $\mcZ(\fkf,\bbP^n_{\bbR})$ is of dimension at most $n-q$ and
\begin{equation}\label{eq:fewnomialvolume}
    \bbE_{\fkf}\,\vol_{n-q}\mcZ(\fkf,\bbP^n_{\bbR})\leq \frac{\vol_{n-q}\bbP^{n-q}_{\bbR}}{\vol_{n}\bbP^{n}_{\bbR}}\,\frac{(n(n+1))^{n-q}}{2^n}\prod_{k=1}^qt_k(t_k-1) 
\end{equation}
\end{theo}
\begin{remark}

In Theorem~\ref{theo:MainTheoremNumberZeros_volume}, we are imposing $\conv(A_k)=d_k\Delta_n$ to guarantee that $\mcZ(\fkf,\bbP^n_{\bbR})$ does not contain a coordinate hyperplane with probability one. Although this can be achieved under more general conditions, we prioritize a simple statement involving the $d_ke_j$.
\end{remark}
\begin{remark}
By our assumption, we have $t_k\geq n+1$. Hence, when degree is large, for example, when
\[
t_k\leq \frac{2^{n/q}}{(n(n+1))^{n/q-1}}\sqrt{d_k},
\]
\eqref{eq:fewnomialvolume} gives a better bound than \eqref{eq:kssvolume}. In particular, the random projective fewnomial hypersurface given by
\[
\sum_{i=0}^n\fkf_iX_i^d
\]
has an average volume bounded by
\[
\frac{\vol_{n-q}\bbP^{n-q}_{\bbR}}{\vol_{n}\bbP^{n}_{\bbR}}\,\left(\frac{n(n+1)}{2}\right)^n,
\]
which does not depend on the degree.
\end{remark}

\subsection{Overview and Outline of the Proof}
Our proof follows a similar path to the proof in~\cite{BETC-fewnomials}. We first prove an integral-probabilistic formula for the number of zeros which depends on the determinant of a random matrix. To bound this complicated integral, we use the Cauchy-Binet formula\footnote{Curiously enough, this formula appears in the proof of many upper bounds in fewnomial theory~\cite{bihansottile2007,burgisser2023,BETC-fewnomials}.} and express the determinant as a summation of many determinants of $n \times n$ matrices. In turn, this splits the integral formula into a sum of integrals over determinants of $n \times n$ matrices. We bound these summands by interpreting them as expected number of real zeros of a specific system of random polynomials and bounding these expectations simply by the maximum number of zeros of these special systems. 

Expanding the above, we prove the integral-probabilistic formula \eqref{eq:intprobformula} in Section~\ref{sec:riceformula} using kinematic formulas \cite{burgisserlerario2020,howard1993} and a measure theoretic result of Appendix~\ref{sec:meeasures}. After this, in Section~\ref{sec:proofMainResult}, we use the Cauchy-Binet formula~\cite[Theorem 4.15]{broidagill1989linearalgebratextbook} to split the integral into several summands. However,  the proof complicates and we need to do a polyhedral subdivision---induced by the normal fan of $\sum_{k=1}^n\mcL(A_k,\liftfunct_k)$ in the statement of Theorem~\ref{theo:maintheorem_fulldetails}---to perform the reduction of the random general case to the deterministic case---a binomial system---analyzed in Proposition~\ref{prop:boundbinomialsystem}.

Although the above reduction from the general random case to special deterministic cases plays a central role in our proof, a direct analytical proof is possible. We provide a fully analytical proof of Proposition~\ref{prop:boundbinomialsystem} to show this fact.

\subsection*{Organization}
In Section~\ref{sec:riceformula} we obtain a Rice formula for the expected number of real zeros. Section~\ref{sec:proofMainResult} proves our main result. Section~\ref{sec:proofMainResultUnmixed} deals with unmixed systems, and Section~\ref{sec:volumes} contains the proof of volume estimate for random fewnomial varieties.

\section{A Rice Formula for the Expected Number of Zeros}\label{sec:riceformula}

At the center of \cite{burgisser2023}, we find an integral-probabilistic formula for the number of expected zeros of a random system. This formula is only implicitly stated in a special case (the case where $\psi$ is injective). We make this formula explicit and prove it in its full generality. We do this by combining the kinematic formulas \cite{burgisserlerario2020,howard1993}  used in \cite{burgisser2023} with standard arguments of measure theory, such as Dynkin's lemma \cite[4.11]{aliprantisborder2006book}.
We note that the formula below is a special case of Rice formula~\cite[Theorem 6.2]{azaiswschebor2009}. We provide a proof using the kinematic formula for completeness and because this proof is interesting in its own right.

Recall that for $f:U\subseteq\bbR^n\rightarrow \bbR^n$ with $U\subseteq \bbR^n$ open,
\[
\mcZ_r(f,U):=\{\zeta\in U\mid f(x)=0,\,\det \diff_xf(x)\neq 0\}
\]
is the set of regular zeros of $f$ in $U$. The following theorem gives the promised integral-probabilistic formula for the number of zeros of a random system where each equation is generated independently. 

\begin{theo}[Rice formula]\label{theo:intprobformula}
Let $\Omega\subseteq \bbR^n$ be an open set and let
\[\varphi_k:\Omega\rightarrow \bbR^{m_k+1}\]
be smooth maps such that the map $\psi:\Omega\rightarrow \prod_{k=1}^n \bbS^{m_k}$ given by
\begin{equation}
    \psi:x\mapsto \begin{pmatrix}\psi_1(x)\\\vdots\\\psi_n(x)\end{pmatrix}:=\begin{pmatrix}\varphi_1(x)/\|\varphi_1(x)\|\\\vdots\\\varphi_n(x)/\|\varphi_n(x)\|\end{pmatrix}
\end{equation}
is well-defined. Consider a random system of equations $\fkf_1(x)=0,\ldots,\fkf_n(x)=0$ given by
\[\fkf_k:=\sum_{i=0}^{m_k+1}\fkc_{k,i}\varphi_{k,i}\]
where the $\fkc_{k,i}$ are i.i.d. standard Gaussian random variables. 
Then, for every Borel measurable set $\tilde{\Omega}\subseteq\Omega$,
\begin{equation}\label{eq:intprobformula}
    \bbE_{\fkc}\#\mcZ_r(\fkf,\tilde{\Omega})=(2\pi)^{-n/2}\int_{\tilde{\Omega}}\,\bbE_{\fka}\left|\det\begin{pmatrix}
    \fka_1^T\diff_x\psi_1\\
    \vdots\\
    \fka_n^T\diff_x\psi_n
\end{pmatrix}\right|\,\mathrm{d}x
\end{equation}
where the $\fka_k$ are independent standard Gaussian vectors in $\psi_k(x)^\perp\subset \bbR^{m_k+1}$, i.e., the subspace orthogonal to $\psi_k(x)$.
\end{theo}

\begin{cor}\label{cor:intprobformula_phiform}
Under the same assumptions and notations of Theorem~\ref{theo:intprobformula}, for every Borel measurable set $\tilde{\Omega}\subseteq\Omega$,
\begin{equation}\label{eq:intprobformula_phiform}
    \bbE_{\fkc}\#\mcZ_r(\fkf,\tilde{\Omega})=(2\pi)^{-n/2}\int_{\tilde{\Omega}}\,\frac{1}{\prod_{k=1}^n\|\varphi_k(x)\|}\bbE_{\fka}\left|\det\begin{pmatrix}
    \fka_1^T\diff_x\varphi_1\\
    \vdots\\
    \fka_n^T\diff_x\varphi_n
\end{pmatrix}\right|\,\mathrm{d}x
\end{equation}
where the $\fka_k$ are independent standard Gaussian vectors in the subspace $\varphi_k(x)^\perp$.
\end{cor}
\begin{cor}\label{cor:intprobformula_toruscase}
Under the same assumptions and notations of Theorem~\ref{theo:intprobformula}, if 
\[m_1=m_2=\cdots=m_n=1\]
and, for some $\xi:\Omega\rightarrow \bbR^n$,
\[
\varphi_1=\begin{pmatrix}
    1\\
    \xi_1(x)
\end{pmatrix},\ldots,\varphi_n=\begin{pmatrix}
    1\\
    \xi_n(x)
\end{pmatrix}
\]
then, for every Borel measurable set $\tilde{\Omega}\subseteq\Omega$,
\begin{equation}\label{eq:intprobformula_toruscase}
    \bbE_{\fkc}\#\mcZ_r(\fkf,\tilde{\Omega})=\frac{1}{\pi^n}\int_{\tilde{\Omega}}\,\frac{\left|\det \diff_x\xi\right|}{\prod_{k=1}^n(1+\xi_k(x)^2)}\,\mathrm{d}x
\end{equation}
\end{cor}

\begin{proof}[Proof of Theorem~\ref{theo:intprobformula}]
We will show later that without loss of generality  we can assume that $\psi:\Omega\rightarrow \prod_{k=1}^n\bbS^{m_k}$ is an \emph{embedding}, i.e., it satisfies:
\begin{enumerate}
    \item[(I1)] $\psi$ is an \emph{immersion}, i.e., for all $x\in \Omega$, $\mathrm{rank}\,\diff_x\psi=n$.
    \item[(I2)] $\psi$ is injective.
\end{enumerate}
Moreover, we can further assume that
\begin{enumerate}
    \item[(O)] $\tilde{\Omega}=\Omega$. 
\end{enumerate}
Under these assumptions, $\psi(\Omega)$ is an open embedded submanifold of $\prod_{k=1}^n\bbS^{m_k}$. Now, consider the random hyperplanes
\[\fkH_k:=\left\{x\in\bbR^{m_k+1}\,\left|\, \sum_{i=0}^{m_{k}}\fkc_{k,i}x_{k,i} = 0\right.\right\}\]
which are uniformly distributed in the corresponding Grassmannian (because the $\fkc_{k,i}$ are i.i.d. standard Gaussians). Observe that
\begin{equation}\label{eq:lineq}
  \#\mcZ(\fkf,\Omega)=\# \psi(\Omega)\cap \left(\fkH_1\times\cdots\times\fkH_n\right),  
\end{equation}
since $\psi(x)\in \fkH_1\times\cdots\times\fkH_n$ if and only if $x\in\mcZ(\fkf,\Omega)$. Moreover, by Sard's Theorem~\cite[\S 2]{milnor1997} (cf. \cite[Proposition A.18]{conditionbook}), the intersection
\[ \psi(\Omega) \cap \left(\fkH_1\times\cdots\times\fkH_n\right)\]
is transversal with probability one, and hence it is zero-dimensional almost surely. Hence, with probability one,
\[
\#\mcZ_r(\fkf,\Omega)=\#\mcZ(\fkf,\Omega)=\# \psi(\Omega) \cap \left(\fkH_1\times\cdots\times\fkH_n\right),
\]
and so
\[
\bbE_{\fkc}\#\mcZ_r(\fkf,\Omega)=\bbE_{\fkH_1,\ldots,\fkH_n}\# \psi(\Omega) \cap \left(\fkH_1\times\cdots\times\fkH_n\right).
\]
Now, by the kinematic formula stated in \cite[Theorem 3.2]{burgisser2023} (cf. \cite{burgisserlerario2020,howard1993}) and \cite[Lemma 3.4.]{burgisser2023}, the equation \eqref{eq:intprobformula} holds for $\tilde{\Omega}=\Omega$.

We now show that we can assume (I1), (I2) and (O) without loss of generality. If (I1) does not hold, then the set 
\[V:=\{x\in\Omega\mid \mathrm{rank}\,\diff_x\psi<n\}\]
is non-empty. Now for a given $x\in V$, note that 
\begin{equation}\label{eq:derivativepsik}
    \diff_x\psi_k=\frac{1}{\|\varphi_k(x)\|}(\bbI-\psi_k(x)\psi_k(x)^T)\diff_x\varphi_k.
\end{equation}
So, to have $\rank\diff_x \psi < n$ we have to have a $v_x\in \bbR^n\setminus 0$  in the kernel of $\diff_x \psi$ for which we have that for all $k$
\[\diff_x\varphi_k(v_x)\in \bbR \varphi_k(x) .\]
That is there exist $v_x\in \bbR^n\setminus 0$ and $t_1,\ldots,t_n\in\bbR$ such that for all $k$,
\[\diff_x\varphi_k(v_x) = t_k \varphi_k(x).\]
Translating this into $\diff_x \fkf$ we conclude that there are $t_k\in \bbR$ that satisfies
\[
\diff_x\fkf (v_x)
=\begin{pmatrix}
    \fkc_1^T\diff_x \varphi_1(v_x)\\
    \vdots\\
    \fkc_n^T\diff_x \varphi_n(v_x)
\end{pmatrix}
=\begin{pmatrix}
    t_1\fkc_1^T\varphi_1(x)\\
    \vdots\\
    t_n\fkc_n^T\varphi_n(x)
\end{pmatrix}
=\begin{pmatrix}
    t_1\fkf_1(x)\\
    \vdots\\
    t_n\fkf_n(x)
\end{pmatrix}.
\]
This means $\fkf$ cannot have regular zeros at $x$ for any $ x \in V$. Hence, 
\[ \#\mcZ_r(\fkf,V)=0 \; \text{and} \; \#\mcZ_r(\fkf,\tilde{\Omega})=\#\mcZ_r(\fkf,\tilde{\Omega}\setminus V) . \]
Moreover, since for all $x\in V$, $\mathrm{rank}\,\diff_x\psi<n$, we have that for all $a_1\in \bbR^{m_1+1},\ldots,a_n\in \bbR^{m_n+1}$,
\[
\det\begin{pmatrix}
    a_1^T\diff_x\psi_1\\
    \vdots\\
    a_n^T\diff_x\psi_n
\end{pmatrix}=0
\]
and so 
\[
\int_{\tilde{\Omega}}\,\bbE_{\fka}\left|\det\begin{pmatrix}
    \fka_1^T\diff_x\psi_1\\
    \vdots\\
    \fka_n^T\diff_x\psi_n
\end{pmatrix}\right|\,\mathrm{d}x=\int_{\tilde{\Omega}\setminus V}\,\bbE_{\fka}\left|\det\begin{pmatrix}
    \fka_1^T\diff_x\psi_1\\
    \vdots\\
    \fka_n^T\diff_x\psi_n
\end{pmatrix}\right|\,\mathrm{d}x.
\]
Hence, taking $\Omega\setminus V$ instead of $\Omega$, we can assume, without loss of generality, that (I1) holds.

Assume now that (I1) holds. We will show how we can assume (I2) and (O) without loss of generality. We will make use of Theorem~\ref{theo:measuretheoryargument}, which relies on standard arguments in measure theory and which we prove in Appendix~\ref{sec:meeasures}. By (I1), the map $\psi:\Omega\rightarrow \bbS^n$ is an immersion and so, by \cite[Proposition 5.22]{lee2013book}, it is a \emph{local embedding}, i.e., for each $x\in \Omega$, there is $r_x>0$ such that $\psi:B(x,r_x)\subseteq \Omega\rightarrow \prod_{k=1}^{n}\bbS^{m_k}$ is injective. Therefore the collection of open sets
\begin{equation}\label{eq:functionalcover}
\mcU:=\{U\subseteq \Omega\mid U\text{ open, }\overline{U}\subseteq \Omega,\,\overline{U}\text{ compact and }\psi_{|U}\text{ injective}\}.
\end{equation}
satisfies that for all $x\in \Omega$, there is $r_x>0$ such that for all $r<r_x$, $B(x,r)\in\mcU$. Hence $\mcU$ is a base for the topology of $\Omega$---assumption (U0) of Theorem~\ref{theo:measuretheoryargument}. Moreover, $\mcU$ is closed under containment of open sets---assumption (U1) of Theorem~\ref{theo:measuretheoryargument}.

Consider the following two Borel measures:
\[
\mu:B\mapsto \bbE_{\fkc}\#\mcZ_r(\fkf,B)
\]
and
\[
\nu:B\mapsto (2\pi)^{-n/2}\int_{B}\,\bbE_{\fka}\left|\det\begin{pmatrix}
    \fka_1^T\diff_x\psi_1\\
    \vdots\\
    \fka_n^T\diff_x\psi_n
\end{pmatrix}\right|\,\mathrm{d}x.
\]
We only need to show that $\mu=\nu$. Once this is done, \eqref{eq:intprobformula} holds in the desired generality.

For each $U\in \mcU$, (I2) holds. Thus we have that \eqref{eq:intprobformula} holds with $\tilde{\Omega}=U$, and so
\[
\mu(U)=\nu(U)
\]
---assumption (U4) of Theorem~\ref{theo:measuretheoryargument}.
Moreover, this quantity is finite---assumption (U3) of Theorem~\ref{theo:measuretheoryargument}. To see this, note that 
\[
\bbE_{\fka}\left|\det\begin{pmatrix}
    \fka_1^T\diff_x\psi_1\\
    \vdots\\
    \fka_n^T\diff_x\psi_n
\end{pmatrix}\right|\leq \sqrt{\bbE_{\fka}\prod_{k=1}^n \|\fka_k\|^2}\sqrt{\bbE_{\fka}\left|\det\begin{pmatrix}
    \hat{\fka}_1^T\diff_x\psi_1\\
    \vdots\\
    \hat{\fka}_n^T\diff_x\psi_n
\end{pmatrix}\right|^2}
\]
where $\hat{\fka}_k:=\fka_k/\|\fka_k\|$ is uniformly distributed in $\bbS^{m_k}\cap \psi_k(x)^\perp$., and so, by compactness, for $x\in U$,
\[
\bbE_{\fka}\left|\det\begin{pmatrix}
    \fka_1^T\diff_x\psi_1\\
    \vdots\\
    \fka_n^T\diff_x\psi_n
\end{pmatrix}\right|\leq \sqrt{\bbE_{\fka}\prod_{k=1}^n \|\fka_k\|^2}\max_{\substack{\hat{a}_k\in\bbS^{m_k}\\x\in\overline{U}}}\left|\det\begin{pmatrix}
    \hat{a}_1^T\diff_x\psi_1\\
    \vdots\\
    \hat{a}_n^T\diff_x\psi_n
\end{pmatrix}\right|<\infty.
\]
Thus the integral of the left-hand side over $U\subset \overline{U}$, and so $\nu(U)$, must be finite.

Now, $\Omega$ is $\sigma$-compact, because it is an open subset of $\bbR^n$ and all hypothesis of Theorem~\ref{theo:measuretheoryargument} ((U0), (U1), (U2) and (U3)) are satisfied, so $\mu=\nu$ as we wanted to show.
\end{proof}
\begin{remark}
The above proof can be applied to extend \cite{BETC-fewnomials} to functions that are not semialgebraic and this immediately extends the result in \cite{BETC-fewnomials} to arbitrary real exponents. 
\end{remark}
\begin{proof}[Proof of Corollary~\ref{cor:intprobformula_phiform}]
This follows from Theorem~\ref{theo:intprobformula} using \eqref{eq:derivativepsik}, the properties of the determinant and that $\fka_k(\bbI-\psi_k(x)\psi_k(x)^T)=\fka_k$ due to $\fka_k\in\psi_k(x)^\perp=\varphi_k(x)^\perp$.
\end{proof}
\begin{proof}[Proof of Corollary~\ref{cor:intprobformula_toruscase}]
We just have to apply Corollary~\ref{cor:intprobformula_phiform}. Fix $x\in \Omega$, we only have to show that
\[(2\pi)^{-\frac{n}{2}}\bbE_{\fka}\left|\det\begin{pmatrix}
    \fka_1^T\diff_x\varphi_1\\
    \vdots\\
    \fka_n^T\diff_x\varphi_n
\end{pmatrix}\right|=\frac{1}{\pi^n}\frac{|\det\diff_x\xi|}{\prod_{k=1}^n\sqrt{1+\xi_k(x)^2)}}\]
where the $\fka_k\in\psi_k(x)^\perp$ are independent standard Gaussian random vectors.

Since we are in $\bbR^2$, we have that
\[
\varphi_k(x)^\perp=\mathrm{span}\,\begin{pmatrix}
    \frac{1}{\sqrt{1+\xi_k(x)}}\begin{pmatrix}
        -\xi_k(x)\\
        1
    \end{pmatrix}
\end{pmatrix}
\]
and so we can write
\[
\fka_k=\frac{\fkz_k}{\sqrt{1+\xi_k(x)}}\begin{pmatrix}
        -\xi_k(x)\\
        1
    \end{pmatrix}
\]
with the $\fkz_k\in\bbR$ i.i.d. standard Gaussians. Moreover, because of this, $\fka_k\diff_x\varphi_k=\frac{\fkz_k}{\sqrt{1+\xi_k(x)^2}}\diff_x\xi_k$, and so
\[
\left|\det\begin{pmatrix}
    \fka_1^T\diff_x\varphi_1\\
    \vdots\\
    \fka_n^T\diff_x\varphi_n
\end{pmatrix}\right|=
\prod_{k=1}^n|\fkz_k|\,\frac{|\det\diff_x\xi|}{\prod_{k=1}^n\sqrt{1+\xi_k(x)^2}}.
\]
Now, an easy computation shows that $\bbE_{\fkz}\prod_{k=1}^n|\fkz_k|=(2/\pi)^{n/2}$, and the claim follows.
\end{proof}

We finish this section with the following two propositions which will allow us to drop the regularity assumption from our theorems.

\begin{prop}\label{prop:avoidregularity}
Under the same assumptions and notations of Theorem~\ref{theo:intprobformula}, if for all $x\in \Omega$,
\[
\mathrm{rank}\,\diff_x\psi=n,
\]
then, for every Borel measurable set $\tilde{\Omega}\subseteq \Omega$,
\[
\mcZ(\fkf,\tilde{\Omega})=\mcZ_r(\fkf,\tilde{\Omega})
\]
with probability one.
\end{prop}
\begin{prop}\label{prop:emptyoverdetermined}
Under the same analogous notations of Theorem~\ref{theo:intprobformula}, assume that the map
the map $\psi:\Omega\rightarrow \prod_{k=1}^q \bbS^{m_k}$, with $q>n$, given by
\begin{equation}
    \psi:x\mapsto \begin{pmatrix}\psi_1(x)\\\vdots\\\psi_q(x)\end{pmatrix}:=\begin{pmatrix}\varphi_1(x)/\|\varphi_1(x)\|\\\vdots\\\varphi_q(x)/\|\varphi_q(x)\|\end{pmatrix}
\end{equation}
is well-defined and that for all $x\in \Omega$,
\[
\mathrm{rank}\,\diff_x\psi=n.
\] 
Consider the random overdetermined system $\fkf_1(x)=0,\ldots,\fkf_q(x)=0$ given by
\[\fkf_k:=\sum_{i=0}^{m_k+1}\fkc_{k,i}\varphi_{k,i}\]
where the $\fkc_{k,i}$ are i.i.d. standard Gaussian random variables. Then
\[
\mcZ(\fkf,\Omega)=\varnothing
\]
with probability one.
\end{prop}
\begin{proof}[Proof of Proposition~\ref{prop:avoidregularity}]
The proof is as the proof of Theorem~\ref{theo:intprobformula}, but now we don't need to remove any subset from $\Omega$ so that that (I1) holds. Hence, arguing as before using~\eqref{eq:lineq} with the cover $\mcU$, we have, by Sard's Theorem~\cite[\S 2]{milnor1997} (cf. \cite[Proposition A.18]{conditionbook}), that for all $U\in \mcU$,
\[
\mcZ(\fkf,U)=\mcZ_r(\fkf,U)
\]
with probability one. Since this holds for all open sets in an open cover of $\Omega$, it holds for $\Omega$ too. 
\end{proof}
\begin{proof}[Proof of Proposition~\ref{prop:emptyoverdetermined}]
The proof is analogue to that of Proposition~\ref{prop:avoidregularity}. However, when writing down the analogous of \eqref{eq:lineq}, we have that we are intersecting $q>n$ hyperplanes with an $n$-dimensional submanifold in $\prod_{k=1}^q \bbS^{m_k}$. By Sard's Theorem~\cite[\S 2]{milnor1997} (cf. \cite[Proposition A.18]{conditionbook}), this intersection is almost surely transversal which implies that this intersection is empty  as desired.
\end{proof}

\section{Proof of the Main Theorem}\label{sec:proofMainResult}

We  now prove Theorem~\ref{theo:maintheorem_fulldetails} which yields Theorems~\ref{theo:MainTheoremNumberZeros} and~\ref{theo:MainTheoremNumberZeros_variances} as direct corollaries. We also prove Proposition~\ref{prop:injectivity}, which allows us to remove the regularity from our statements for random fewnomial systems; and the key Proposition~\ref{prop:boundbinomialsystem} that plays a central role in the proof of the main theorem.

\begin{prop}\label{prop:injectivity}
Under the notations and assumptions of Theorem~\ref{theo:maintheorem_fulldetails}, 
\[
\mcZ(\fkf,\bbR^n_+)=\mcZ_r(\fkf,\bbR_+^n)
\]
with probability one.
\end{prop}
\begin{prop}\label{prop:boundbinomialsystem}
Let $\gamma_1,\ldots,\gamma_n\in\bbR^n$, $s_1,\ldots,s_n\in\bbR$ and consider the following the random binomial system $\fkh$ given by the equations
\begin{equation}\label{eq:specialsystem}
\left.
    \begin{array}{rl}
        \fkh_1&=\fka_1\exp(\gamma_1^tX+s_1)+\fkb_1\\
        &\,\vdots\\
        \fkh_{n}&=\fka_n\exp(\gamma_n^tX+s_n)+\fkb_n\\
    \end{array}
\right.
\end{equation}
where the $\fka_k$ and $\fkb_k$ are i.i.d. standard Gaussian random variables, and the following Borel-measurable set
\[
B:=\{x\in\bbR^n\mid \text{for all }k,\,\gamma_k^Tx+s_k\leq 0\}.
\]
Then
\[
\frac{1}{\pi^n}\int_B\,\frac{\prod_{k=1}^n\exp(\gamma_k^Tx+s_k)}{\prod_{k=1}^n(1+\exp(2\gamma_k^Tx+2s_k))}\left|\det(\gamma_1\,\cdots\,\gamma_n)\right|\,\mathrm{d}x=\bbE_\fkh \#\mcZ_r(\fkh,B)\leq \frac{1}{4^n}.
\]
\end{prop}
\begin{remark}
If the $\gamma_k$ are linearly independent, then the inequality in Proposition~\ref{prop:boundbinomialsystem} is an equality.
\end{remark}

\begin{proof}[Proof of Theorem~\ref{theo:maintheorem_fulldetails}]
By Proposition~\ref{prop:injectivity}, we can just bound the number of expected regular zeros. 
Instead of a system of fewnomials, we consider the equivalent system of exponential sums $\fkg=(\fkg_1,\ldots,\fkg_n)$ defined as
\begin{equation}
    \begin{array}{rl}
        \fkg_1&=\sum_{\alpha\in A_1}\fkc_{1,\alpha}\exp(\alpha^TX+\liftfunct_1(\alpha))\\
        &\,\vdots\\
        \fkg_n&=\sum_{\alpha\in A_n}\fkc_{n,\alpha}\exp(\alpha^TX+\liftfunct_n(\alpha))\\
    \end{array}
\end{equation}
where the $\fkc_{k,\alpha}$ are i.i.d. standard Gaussian variables and $\liftfunct_k:A_k\rightarrow \bbR$ is the lifting function defined in \eqref{eq:liftingfunction}, for which $\exp(\liftfunct_k(\alpha))$ equals the square root of the variance of $\fkf_{k,\alpha}$. Observe that 
\[
\bbE_{\fkc}\#\mcZ_r(\fkg,\bbR^n)=\bbE_{\fkf}\#\mcZ_r(\fkf,\bbR^n_+) .
\]
So we will need to bound the expected number zeros of $\fkg$ on $\mathbb{R}^n$. We do this by decomposing  $\mathbb{R}^n$ into a collection of polyhedral cells, and bounding the number the number zeros on each of these polyhedral cells. 

To define the polyhedral cell, for $\alpha_1\in A_1,\ldots,\alpha_n\in A_n$, consider 
\begin{equation}
M(\alpha_1,\ldots,\alpha_n):=\{x\in\bbR^n\mid  \; \text{for all }k \; \text{and for all } \alpha\in A_k,\, 
 \alpha^Tx +\liftfunct_k(\alpha) \leq  \alpha_k^T x +\liftfunct_k(\alpha_k) 
  \}
\end{equation}
Observe that $z \in M(\alpha_1,\ldots,\alpha_n)$ if and only if the map 
\[
\begin{pmatrix}s\\x\end{pmatrix}\mapsto \begin{pmatrix}1\\z\end{pmatrix}^T\begin{pmatrix}s\\x\end{pmatrix}=s+z^tx
\]
attains, for each $k$, its maximum value inside 
$\mcL(A_k,\liftfunct_k)$
at
\[
\begin{pmatrix}\liftfunct_k(\alpha_k)\\\alpha_k\end{pmatrix},
\]
i.e.,  
\[
\text{ if }\begin{pmatrix}1\\z\end{pmatrix}
\text{ is, for each }k\text{, in the normal cone of }\mcL_k(A_k,\liftfunct_k)\text{ at }\begin{pmatrix}\liftfunct_k(\alpha_k)\\\alpha_k\end{pmatrix}.
\]
This shows that $M(\alpha_1,\ldots,\alpha_n)$ is a full-dimensional polyhedral cell if and only if 
\[
\sum_{k=1}^n\begin{pmatrix} \liftfunct_k(\alpha_k)\\\alpha_k\end{pmatrix}
\]
is a vertex of the Minkowski sum $\sum_{k=1}^n\mcL(A_k,\liftfunct_k)$. Moreover, since
$M(\alpha_1,\ldots,\alpha_n)$ is obtained by maximizing the value of linear functional over a finite collection of points, we have
\begin{equation}\label{eq:decompositionMbetas}
   \bbR^n=\bigcup\{M(\alpha_1,\ldots,\alpha_n)\mid \alpha_1\in A_1,\ldots,\alpha_n\in A_n\}. 
\end{equation}
\begin{figure}[bt]
    \centering
    \includegraphics[width=0.8\textwidth]{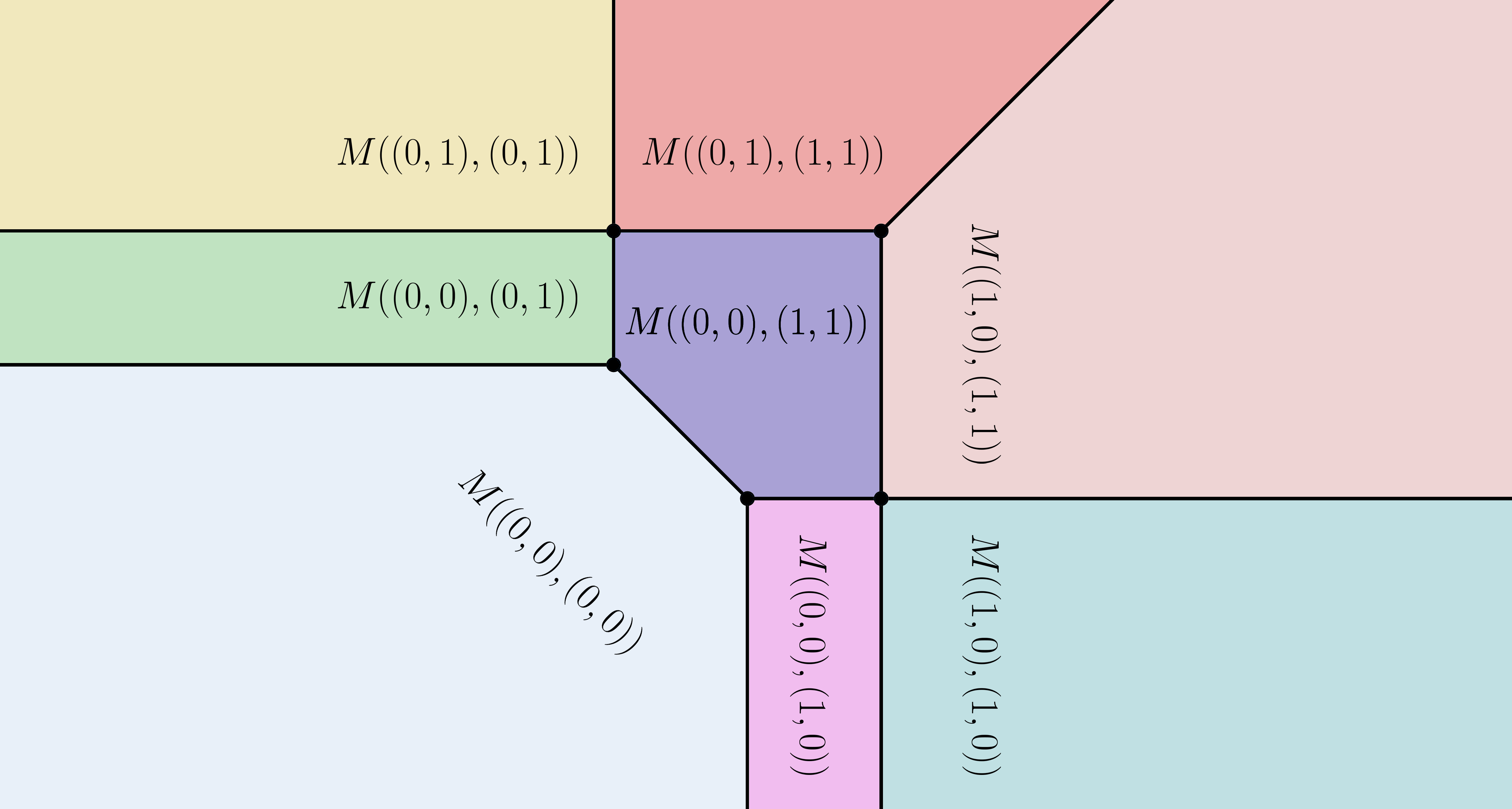}
    \caption{The collection of full-dimensional $M(\alpha_1,\alpha_2)$ for the $(A_1,\liftfunct_1)$ and $(A_2,\liftfunct_2)$ of Figure~\ref{fig:exam1}.}
    \label{fig:exam2}
\end{figure}
So, by subadditivity of the zero counting function, we only need to prove that for all $\alpha_1\in A_1,\ldots,\alpha_n\in A_n$ the following bound holds
\begin{equation}
   \bbE_{\fkc}\#\mcZ_r(\fkg,M(\alpha_1,\ldots,\alpha_n))\leq \frac{1}{4^n}\prod_{k=1}^n(t_k-1). 
\end{equation}
Observe that the solutions of $\fkg$ in any region do not change if we multiply each equation by $\exp \left( - \alpha_k^Tx -\liftfunct_k(\alpha_k) \right)$. Thus, without loss of generality, we can assume that $\alpha_1=\cdots=\alpha_n=0$ and that $\liftfunct_k(\alpha_k)=0$. Thus by considering $M:=M(0,\ldots,0)$, $ \varphi_k(x)=\left( \exp(\alpha^Tx+\liftfunct_{i}(\alpha)) \right)_{\alpha \in A_k}$, and applying Corollary~\ref{cor:intprobformula_phiform}, we get
\begin{equation}\label{eq:boundexpectationgonM}
    \bbE_{\fkc}\#\mcZ_r(\fkg,M)=(2\pi)^{-\frac{n}{2}}\int_M\,\frac{1}{\prod_{k=1}^n\|\varphi_k(x)\|}\bbE_{\fka}\left|\det\begin{pmatrix}
    \fka_1^T\diff_x\varphi_1\\
    \vdots\\
    \fka_n^T\diff_x\varphi_n
\end{pmatrix}\right|\,\mathrm{d}x
\end{equation}
where the $\fka_k\in\varphi_k(x)^\perp$ are independent standard Gaussian random vectors. Observe that $\diff_x\varphi_{k,0}=0$ since $\varphi_{k,0}=1$. Thus, by the Cauchy-Binet formula \cite[Theorem 4.15]{broidagill1989linearalgebratextbook} applied to
\[
\begin{pmatrix}
    \fka_1^T\diff_x\varphi_1\\
    \vdots\\
    \fka_n^T\diff_x\varphi_n
\end{pmatrix}=\begin{pmatrix}
    \fka_1^T&&\\&\ddots&\\&&\fka_n^T
\end{pmatrix}\diff_x\varphi,
\]
we obtain 
\[
\det\begin{pmatrix}
    \fka_1^T\diff_x\varphi_1\\
    \vdots\\
    \fka_n^T\diff_x\varphi_n
\end{pmatrix}
= \sum_{\beta_1\in A_1\setminus\{0\},\ldots,\beta_n\in A_n\setminus\{0\}}\prod_{k=1}^n\fka_{k,\beta_k}\det\begin{pmatrix}
    \diff_x\varphi_{1,\beta_1}\\
    \vdots\\
    \diff_x\varphi_{n,\beta_n}
\end{pmatrix},
\]
since, to get a full-rank minor of $\mathrm{diag}(\fka^T_k)$, we need to pick a column corresponding to each $\fka_k^T$. So, by the triangle inequality we have
\begin{equation}\label{eq:CBineq}
    \left|\det\begin{pmatrix}
    \fka_1^T\diff_x\varphi_1\\
    \vdots\\
    \fka_n^T\diff_x\varphi_n
\end{pmatrix}\right|
\leq \sum_{\beta_1\in A_1\setminus\{0\},\ldots,\beta_n\in A_n\setminus\{0\}}\prod_{k=1}^n|\fka_{k,\beta_k}|\left|\det\begin{pmatrix}
    \diff_x\varphi_{1,\beta_1}\\
    \vdots\\
    \diff_x\varphi_{n,\beta_n}
\end{pmatrix}\right|.
\end{equation}
Now we use following observation: The coordinate projection
\begin{align*}
    \bbR^{A_k}&\rightarrow\bbR\\
    x&\mapsto x_{\beta_k}
\end{align*}
induces a linear functional $\lambda_{k}:\varphi_k(x)^\perp \rightarrow \bbR$ which has the norm
\begin{equation*}
    \|\lambda_{k}\|=\sqrt{1-\frac{\varphi_{k,\beta_k}(x)^2}{\|\varphi_k(x)\|^2}}=\frac{\sqrt{\|\varphi_k(x)\|^2-\varphi_{k,\beta_k}(x)^2}}{\|\varphi_k(x)\|}.
\end{equation*}
Thus $\fka_{k,\beta_k}=\lambda_{\beta_k}(\fka_k)$ is a centered Gaussian random variable with variance $\|\lambda_{k}\|^2$, because it is the result of projecting the standard Gaussian random vector $\fka_k\in\varphi_k(x)^\perp$ with the linear functional $\lambda_{k}$ of norm $\|\lambda_{k}\|$. Therefore
\begin{equation}\label{eq:expak}
    \bbE_{\fka}|\fka_{k,\beta_k}|=\frac{\sqrt{\|\varphi_k(x)\|^2-\varphi_{k,\beta_k}(x)^2}}{\|\varphi_k(x)\|}\sqrt{\frac{2}{\pi}}.
\end{equation}
Substituting \eqref{eq:expak} and $\diff_x\varphi_{k,\beta_k}=\exp(\beta_k^Tx +\liftfunct_k(\beta_k))\beta_k^T$ back in~\eqref{eq:boundexpectationgonM} combined with \eqref{eq:CBineq}, we get
\begin{multline*}
   \bbE_\fkc\#\mcZ_r(\fkg,M)\\\leq \sum_{\substack{\text{for all }k\text{, }\\\beta_k\in A_k\setminus \{0\}}}\frac{1}{\pi^n}\int_M\,\prod_{k=1}^n\frac{\sqrt{\|\varphi_k(x)\|^2-\varphi_{k,\beta_k}(x)^2}\exp(\beta_k^Tx +\liftfunct_k(\beta_k))}{\|\varphi_k(x)\|^2}\,|\det(\beta_1\,\cdots\,\beta_n)|\,\mathrm{d}x. 
\end{multline*}
If we bound each of the summands in the right-hand side by $1/4^n$, we are done, since there are at most $\prod_{k=1}^n(t_k-1)$ many summands.

We will use Proposition~\ref{prop:boundbinomialsystem} with $\gamma_k=\beta_k$ and $s_k=\liftfunct_k(\beta_k)$ to bound each summand by $1/4^n$. To do this, we only need to show that
\begin{equation}\label{eq:desiredineq}
\prod_{k=1}^n\frac{\sqrt{\|\varphi_k(x)\|^2-\varphi_{k,\beta_k}(x)^2}}{\|\varphi_k(x)\|^2}\leq \frac{1}{\prod_{k=1}^n(1+\exp(2 \beta_k^Tx +2 \liftfunct_k(\beta_k)))}.
\end{equation}
So, it suffices to show that for each $k$ we have
\begin{equation}
    \frac{\sqrt{\|\varphi_k(x)\|^2-\varphi_{k,\beta_k}(x)^2}}{\|\varphi_k(x)\|^2}\leq \frac{1}{(1+\exp(2 \beta_k^Tx + 2 \liftfunct_k(\beta_k)))}.
\end{equation}
We define
\[
u_k(x):=\varphi_{k,\beta_k}(x)=\exp( \beta_k^Tx +\liftfunct_k(\beta_k))
\]
and
\[
v_k(x)=\sqrt{\|\varphi_k(x)\|^2-\varphi_{k,\beta_k}(x)^2}=\sqrt{\sum_{\alpha\in A_k\setminus\{0,\beta_k\}}\exp(2 \alpha^Tx +2 \liftfunct_k(\alpha))}.
\]
We only need to show then that
\[
\frac{\sqrt{1+v_k(x)^2}}{1+u_k(x)^2+v_k(x)^2}\leq \frac{1}{1+u_k(x)^2},
\]
which is equivalent to
\[
\left(\sqrt{1+v_k(x)^2}\right)^2-(1+u_k(x)^2)\sqrt{1+v_k(x)^2}+u_k(x)^2\geq 0.
\]
Observe that the polynomial $T^2-(1+u_k(x)^2)T+u_k(x)^2$ has two roots: $1$ and $u_k(x)^2$. Since $\sqrt{1+v_k(x)^2}\geq 1$, the above inequality is valid independently of $v_k(x)$ as long as $u_k(x)\leq 1$, which is precisely what happens for $x \in M$ since
\[
M=M(0,\ldots,0)=\{x\in\bbR^n\mid \text{for all } k \text{, for any }\beta_k \in A_k,\, \beta_k^Tx + \liftfunct_k(\beta_k)\leq 0\}.
\]
Hence, \eqref{eq:desiredineq} holds in $M$, and by Proposition~\ref{prop:boundbinomialsystem},
\[
\frac{1}{\pi^n}\int_M\,\prod_{k=1}^n\frac{\sqrt{\|\varphi_k(x)\|^2-\varphi_{k,\beta_k}(x)^2}\exp( \beta_k^Tx +\liftfunct_k(\beta_k))}{\|\varphi_k(x)\|^2}\,|\det(\beta_1\,\cdots\,\beta_n)|\,\mathrm{d}x\leq \frac{1}{4^n}.
\]
which completes the proof.
\end{proof}

\begin{proof}[Proof of Proposition~\ref{prop:injectivity}]
Assume that the affine span of $\sum_{k=1}^n A_k$ is the whole of $\bbR^n$. By Propositions~\ref{prop:avoidregularity}, we only need to show that for the map
\[ \psi=(\varphi_1/\|\varphi_1\|,\ldots,\varphi_n/\|\varphi_n\|) \] 
where 
\[\varphi_k(x)=(\exp( \alpha^Tx+ \liftfunct_k(\alpha)))_{\alpha\in A_k},\]
satisfies that for all $x\in\bbR^n$, $\mathrm{rank}\,\diff_x\psi=n$. Note that working with this exponential formulation is enough as it does not affect the regularity of the zeros of a fewnomial system.

Now, by \eqref{eq:derivativepsik}, we have that $\mathrm{rank}\,\diff_x\psi<n$ if and only if there is $v_x\in\bbR^n\setminus\{0\}$ such that for all $k$, $\diff_x\varphi(v_x)\in\bbR\varphi(x)$. Thus $\mathrm{rank}\,\diff_x\psi<n$ if and only if there is $v_x\in\bbR^n\setminus\{0\}$ and $t_1,\ldots,t_n\in\bbR^n$ such that for all $k$, $\diff_x\varphi(v_x)=t_k\varphi_k(x)$. For all $k$ and $\alpha_k \in A_k$,
\[
\diff_x\varphi_{k,\alpha_k}=\varphi_{k,\alpha_k}\alpha_k^T.
\]
Hence $\mathrm{rank}\,\diff_x\psi<n$ if and only if there is $v_x\in\bbR^n\setminus\{0\}$ and $t_1,\ldots,t_n\in\bbR^n$ such that for all $k$ and all $\alpha_k\in A_k$,
$
\alpha_k^Tv_x=t_k
$.
In other words, $\mathrm{rank}\,\diff_x\psi<n$ if and only if the $A_k$ are contained in parallel (affine) hyperplanes.  If the affine span of $\sum_{k=1}^n A_k$ is $\bbR^n$, as we are assuming, then the latter is not possible and so $\diff_x\psi$ is injective, as desired. 

Now, assume that the affine span of $\sum_{k=1}^n A_k$ is not all of $\bbR^n$. After a change of variables of the form $x\mapsto \exp(A\log(x)+b)$, we can assume, without loss of generality, that the affine span of $\sum_{k=1}^n A_k$ is
\[
\bbR^m\times 0\subset \bbR^n
\]
with $m<n$. In this case we have that $\fkf$ is a random overdetermined system in the variables $X_1,\ldots,X_m$. Moreover, arguing as before, we have that for all $x\in \bbR^m$, $\mathrm{rank}\,\diff_x\psi=m$. Hence, by Proposition~\ref{prop:emptyoverdetermined}, we have
\[
\mcZ(\fkf,\bbR^m_+\times \{\mathds{1} \})=\varnothing
\]
with probability one. This implies that if the affine span of $\sum_{k=1}^n A_k$ is not all of $\bbR^n$  we have $ \mcZ (\fkf,\bbR^n_+)=\varnothing$ with probability one, and so $\mcZ(\fkf,\bbR^n_+)=\mcZ_r(\fkf,\bbR^n_+)$ with probability one.
\end{proof}

\begin{proof}[Geometric Proof of Proposition~\ref{prop:boundbinomialsystem}]

Without loss of generality, assume that the $\gamma_k$ are linearly independent. Otherwise the proposition is immediate, with both the integral and the expectation being zero.

Note that the the left-hand side equality follows from Corollary~\ref{cor:intprobformula_toruscase}, after taking $\xi_k(x)=\exp(\gamma_k^Tx+s_k)$ and observing that $\diff_x\xi_k=\exp(\gamma_k^Tx+s_k)\gamma_k^T$. 
Hence we only need to bound the expectation of the number of real zeros.

Observe that the system~\eqref{eq:specialsystem} have a solution if and only if for all $k$, $\fka_k$ and $\fkb_k$ have opposite signs. Moreover, in that case, the system is equivalent to the following linear system
\begin{equation}
    \left.
    \begin{array}{rl}
        \gamma_1^tX&=-s_1+\ln(-\fkb_1/\fka_1)\\
        \vdots&\\
        \gamma_n^tX&=-s_n+\ln(-\fkb_n/\fka_n)
    \end{array}
\right.
\end{equation}
which has exactly one regular real solution. This system has this unique solution inside $B$ if and only if the $\ln(-\fkb_k/\fka_k)$ are non-positive, which happens with probability $1/4^n$. Hence
\[
\bbE_\fkh \#\mcZ_r(\fkh,B)= \bbP(\#\mcZ_r(\fkh,B)\neq 0) = \frac{1}{4^n},
\]
as we wanted to show.
\end{proof}

\begin{proof}[Analytical Proof of Proposition~\ref{prop:boundbinomialsystem}]

The analytical proof is like the geometric one, but we compute directly now the left-hand side integral. Under the change of variables $y_k=-(\gamma_k^Tx_k+s_k)$, the left-hand side integral becomes
\[
\frac{1}{\pi^n}\int_{\bbR^n_+}\,\frac{\prod_{k=1}^n\exp(-y_k)}{\prod_{k=1}^n(1+\exp(-2y_k))}\,\mathrm{d}y.
\]
Under the further change of variables $z_k=\exp(-y_k)$, this integral becomes
\[
\frac{1}{\pi^n}\int_{[0,1]^n}\,\frac{1}{\prod_{k=1}^n(1+z_k^2)}\,\mathrm{d}z.
\]
Now, separating variables, this integral equals
\[
\left(\frac{1}{\pi}\int_{0}^1\,\frac{1}{1+t^2}\,\mathrm{d}t\right)^n=\frac{1}{4^n},
\]
where the last equality follows from $\arctan'(t)=1/(1+t^2)$.
\end{proof}

\section{Zeros of Random Unmixed Fewnomials}\label{sec:proofMainResultUnmixed}

We now prove Theorem~\ref{theo:MainTheoremNumberZeros_variancesunmixed}. We will actually prove a more general result, and then show that Theorem~\ref{theo:MainTheoremNumberZeros_variancesunmixed} is obtained as a corollary. 

\begin{theo}\label{theo:maintheorem_unmixedfulldetails}
Under the same notations and assumptions of Theorem~\ref{theo:maintheorem_fulldetails}, assume that
\[
A=A_1=\cdots=A_n
\]
and that the Minkowski sum $\sum_{k=1}^n\mcL(A,\liftfunct_k)$ satisfies
\begin{enumerate}
    \item[(MV)] for some $\liftfunct:A\rightarrow \bbR$, $\sum_{k=1}^n\mcL(A,\liftfunct_k)=n\mcL(A,\liftfunct)$.
\end{enumerate}
Then
\begin{equation}
    \bbE_{\fkf}\#\mcZ_r(\fkf,\bbR^n_+)\leq \frac{n+1}{4^n}\binom{V\left(\mcL(A,\liftfunct)\right)}{n+1}
\end{equation}
where $V\left(\mcL(A,\liftfunct)\right)\leq \# A$ is the number of vertices of $\mcL(A,\liftfunct)$.
\end{theo}

\begin{proof}[Proof of Theorem~\ref{theo:maintheorem_unmixedfulldetails}]
The proof is the same as that of Theorem~\ref{theo:maintheorem_fulldetails}, but it differs in the combinatorics at two points: 

\emph{1st point}: When we write the decomposition \eqref{eq:decompositionMbetas}, we only care about those $\alpha_1,\ldots,\alpha_n$ such that
\[
\sum_{k=1}^n\begin{pmatrix} \liftfunct_k(\alpha_k)\\\alpha_k \end{pmatrix}
\]
is a vertex of $\sum_{k=1}^n\mcL(A,\liftfunct_k)$. However, due to the assumption (MV), we necessarily have that  $\alpha_1=\cdots=\alpha_n$, since those are the vertices of $\mcL(A,\liftfunct)$. Hence we only need to look at polyhedral cells of the form $M(\alpha,\ldots,\alpha)$ where $\alpha \in A$ is a vertex of $\mcL(A,\liftfunct)$---and so $n\alpha$ a vertex of $n\mcL(A,\liftfunct)$. Moreover, as we did in the proof of Theorem~\ref{theo:maintheorem_fulldetails}, we can assume, without loss of generality, that $\alpha=0$ and that for all $k$, $\liftfunct_k(\alpha)=0$. 

\emph{2nd point}: When we apply the Cauchy-Binet formula to obtain our decomposition. Since we use the same functions but different variances, we can write for each $k$,
\[
\varphi_k=\Sigma_k \tilde{\varphi}
\]
where $\tilde{\varphi}(x)=\begin{pmatrix}x^\alpha\end{pmatrix}_{\alpha\in A}$ and $\Sigma_k$ is a diagonal positive matrix with entries given by $\exp(\liftfunct_k(\alpha))$. 

 It is important to note that the claim ``we use the same function but different variances'' really requires the extra assumption $\alpha=\alpha_1=\cdots=\alpha_n$. Otherwise, when doing the translation to put $\alpha_1,\ldots,\alpha_n$ at the origin, we will end up with $\varphi_1,\ldots,\varphi_n$ being different not only up to multiplication by a diagonal positive matrix. However, we do not need the $\liftfunct_k(\alpha)$ being independent of $k$, since turning them into zero only requires multiplying the $\varphi_k$ by positive diagonal matrices. Now,
\[
\begin{pmatrix}
    \fka_1^T\diff_x\varphi_1\\
    \vdots\\
    \fka_n^T\diff_x\varphi_n
\end{pmatrix}=\begin{pmatrix}
    \fka_1^T&&\\&\ddots&\\&&\fka_n^T
\end{pmatrix}\begin{pmatrix}
    \Sigma_1\\\vdots\\\Sigma_n
\end{pmatrix}\diff_x\tilde{\varphi}=\begin{pmatrix}
    (\Sigma_1\fka_1)^T\\\vdots\\(\Sigma_n\fka_n)^T
\end{pmatrix}\diff_x\tilde{\varphi},
\]
and so, by the Cauchy-Binet formula~\cite[Theorem 4.15]{broidagill1989linearalgebratextbook} and the triangle inequality,
\[
\left|\det\begin{pmatrix}
    \fka_1^T\diff_x\varphi_1\\
    \vdots\\
    \fka_n^T\diff_x\varphi_n
\end{pmatrix}\right|
\leq \sum_{\{\beta_1,\ldots,\beta_n\}\subset\in A\setminus\{0\}}\prod_{k=1}^n\exp(\liftfunct_k(\beta_k))|\fka_{k,\beta_k}|\left|\det\begin{pmatrix}
    \diff_x\tilde{\varphi}_{\beta_1}\\
    \vdots\\
    \diff_x\tilde{\varphi}_{\beta_n}
\end{pmatrix}\right|,
\]
where the sum runs over subsets of size $n$. Hence we get
\[
\binom{t-1}{n}
\]
summands at most, instead of $(t-1)^n$. Now, for each summand, the same computation as before gives us that each summand becomes
\[
\frac{1}{\pi^n}\int_M\,\prod_{k=1}^n\frac{\sqrt{\|\varphi_k(x)\|^2-\varphi_{k,\beta_k}(x)^2}\exp( \beta_k^Tx+\liftfunct_k(\beta_k))}{\|\varphi_k(x)\|^2}\,|\det(\beta_1\,\cdots\,\beta_n)|\,\mathrm{d}x
\]
and so the proof ends as it did in the proof of Theorem~\ref{theo:maintheorem_fulldetails}. 
\end{proof}

The following lemma shows that Theorem \ref{theo:MainTheoremNumberZeros_variancesunmixed} is a corollary of Theorem \ref{theo:maintheorem_unmixedfulldetails}.

\begin{lem}\label{lem:UVVVimpliesMV}
Under the same notations and assumptions of Theorem~\ref{theo:maintheorem_fulldetails}, assume that $A=A_1=\cdots=A_n$. Then the following holds:
\begin{enumerate}
    \item[(1)] If $\liftfunct=\liftfunct_1=\cdots=\liftfunct_n$ then $\mcL(A,\liftfunct_1)=\cdots=\mcL(A,\liftfunct_n)$ and $\sum_{k=1}^n \mcL(A,\liftfunct_k) = n \mcL(A,\liftfunct)$
    \item[(2)] If for all $k$ and all $\alpha \in A$ we have $\liftfunct_k(\alpha) \leq 0$ with equality whenever $\alpha$ is a vertex of $P:=\conv(A)$, then, for all $k$, 
       \[ \mcL(A,\liftfunct_k)=(-\infty,0]\times P  \]
    and  so $\sum_{k=1}^n\mcL(A,\liftfunct_k)=(-\infty,0]\times \left(n\,\conv(A)\right)=n\mcL(A,0)$. 
\end{enumerate}
\end{lem}
\begin{proof}[Proof of Lemma~\ref{lem:UVVVimpliesMV}]
(1) For a convex set $K \subseteq \mathbb{R}^n$ and $\alpha_1,\alpha_2,\ldots,\alpha_n \in K$
\[
\alpha_1+\cdots+\alpha_n = \frac{1}{n}(n\alpha_1+\cdots+n\alpha_n)\in nK,
\]
and so taking Minkowski sums of $K$ with itself is the same as taking integer dilations of $K$. This is essentially the first claim.

(2) For second claim, we have that, for all $k$, $\liftfunct_k \leq 0$ and equality happens at the vertices of $P$. Thus, for all $k$, $\mcL(A,\liftfunct_k)=(-\infty,0]\times P$. The rest follows as in (1).
\end{proof}

We finish with a proof of Corollary~\ref{cor:circuitcase}.

\begin{proof}[Proof of Corollary~\ref{cor:circuitcase}]
By Markov's inequality~\cite[Corollary 2.9]{conditionbook} and Theorem~\ref{theo:MainTheoremNumberZeros_variancesunmixed},
\[
\bbP_{\fkf}\left(\mcZ(\fkf,\bbR^n_+)\neq \varnothing \right)=\bbP_{\fkf}\left(\#\mcZ(\fkf,\bbR^n_+)\geq 1\right)\leq \bbE_{\fkf}\#\mcZ(\fkf,\bbR^n_+)\leq \frac{n+1}{4^n}\binom{n+\ell}{n+1},
\]
since $\#\mcZ(\fkf,\bbR^n_+)$ is a random variable with integer values. Now,
\[
(n+1)\binom{n+\ell}{n+1}= \ell \binom{n+\ell}{\ell}=\ell \prod_{k=1}^{\ell}\left(\frac{n}{k}+1\right)\leq \ell(n+\ell)^\ell.
\]
Hence the claim follows.
\end{proof}

\section{Volume of Random Projective Fewnomial Varieties}\label{sec:volumes}

To prove Theorem~\ref{theo:MainTheoremNumberZeros_volume}, we will need the following proposition.

\begin{prop}\label{prop:withsimplex}
Let $d_1,\ldots,d_{n+1}\in \bbN$ and $A_1,\ldots,A_{n+1}\subset \bbN^{n+1}$ be finite subsets such that for all $k$,
\[\{d_k e_0,\ldots,d_k e_n\}\subseteq A_k\subseteq d_k\Delta_n\]
where $\{e_0,\ldots,e_n\}$ is the standard basis of $\bbR^{n+1}$ and $\Delta_n:=\conv(e_0,\ldots,e_n)$ the $n$-simplex. Consider the random overdetermined homogeneous fewnomial system $\fkg$ given by
\begin{equation*}
\left.
    \begin{array}{rl}
        \fkg_1&=\sum_{\alpha\in A_1}\fkg_{1,\alpha}X^\alpha\\
        &\,\vdots\\
        \fkg_{n+1}&=\sum_{\alpha\in A_{n+1}}\fkg_{{n+1},\alpha}X^\alpha\\
    \end{array}
\right.
\end{equation*}
where the $\fkg_{k,\alpha}$ are independent continuous random variables. Then 
\[\mcZ(\fkg,\bbP^n_{\bbR})=\varnothing\]
with probability one.
\end{prop}
\begin{proof}[Proof of Theorem~\ref{theo:MainTheoremNumberZeros_volume}]
Consider a homogeneous fewnomial system $f$ given by
\begin{equation*}
\left.
    \begin{array}{rl}
        f_1&=\sum_{\alpha\in A_1}f_{1,\alpha}X^\alpha\\
        &\,\vdots\\
        f_q&=\sum_{\alpha\in A_q}f_{q,\alpha}X^\alpha\\
    \end{array}
\right.
\end{equation*}
Then, by \cite[(2.7)]{matherthomtheorybook}, $\mcZ(f,\bbP^n_{\bbR})$ admits a finite Whitney stratification. Thus, we can consider a partition of $\mcZ(f,\bbP^n_{\bbR})$ into disjoint subsets
\[\mcZ_0(f,\bbP^n_{\bbR}),\ldots,\mcZ_n(f,\bbP^n_{\bbR})\subset \bbP^n_{\bbR}\]
such that that $\mcZ_\ell(f,\bbP^n_{\bbR})$ is an $\ell$-dimensional smooth submanifold of $\mcZ(f,\bbP^n_{\bbR})$ or empty.

Let $\fkH_1,\ldots,\fkH_{n-q}$ be random hyperplanes of $\bbP^n_{\bbR}$ given by
\[
\fkl_k=\sum_{i=0}^n\fka_{k,i}X_i
\]
with the $\fka_{k,i}$ i.i.d. standard Gaussians---one can easily see that this is equivalent to moving a fixed $(n-q)$-plane $L$ in $\bbP^{n}_{\bbR}$ via a random $\fkg\in O(n+1)$ (with respect the Haar probability measure). By \cite[Proposition A.18]{conditionbook}, all the intersections
\begin{equation}\label{eq:linearintersection}
    \mcZ_\ell(f,\bbP^n_\bbR)\cap \fkH_1\cap\cdots \cap \fkH_{n-q}
\end{equation}
are transversal with probability one. Hence, for $\ell<n-q$, \eqref{eq:linearintersection} is empty with probability one; and for $\ell\geq n-q$, \eqref{eq:linearintersection} is, with probability one, either empty or a smooth submanifold of dimension $\ell-(n-q)$. By Poincaré's kinematic formula~\cite[Theorem A.55]{conditionbook}, we have that
\[
\vol_{n-q}\mcZ_{n-q}(f,\bbP^n_\bbR)=\frac{\vol_{n-q}\bbP^{n-q}_{\bbR}}{\vol_{n}\bbP^{n}_{\bbR}}\,\bbE_{\fkH_1,\ldots,\fkH_{n-q}}\#(\mcZ_{n-q}(f,\bbP^n_\bbR)\cap \fkH_1\cap\cdots \cap \fkH_{n-q}).
\]
Now, set the convention that $\vol_{n-q}(\mcZ(f,\bbP^n_\bbR))=\infty$, if for some $\ell>n-q$, $\mcZ_{\ell}(f,\bbP^n_\bbR)\neq \varnothing$; and $\mcZ_{n-q}(f,\bbP^n_\bbR)=0$, if for all $\ell\geq n-q$, $\mcZ_{\ell}(f,\bbP^n_\bbR)= \varnothing$. Then, we have that
\[
\vol_{n-q}\mcZ(f,\bbP^n_\bbR)=\frac{\vol_{n-q}\bbP^{n-q}_{\bbR}}{\vol_{n}\bbP^{n}_{\bbR}}\,\bbE_{\fkH_1,\ldots,\fkH_{n-q}}\#(\mcZ(f,\bbP^n_\bbR)\cap \fkH_1\cap\cdots \cap \fkH_{n-q}).
\]
On the one hand, if some for some $\ell>n-q$, $\mcZ_{\ell}(f,\bbP^n_\bbR)$ is non-empty, then the right-hand side will be infinite---\eqref{eq:linearintersection} will be a non-empty positive dimensional smooth submanifold with probability one. On the other hand, if for all $\ell\geq n-q$, $\mcZ_{\ell}(f,\bbP^n_\bbR)$ is empty, then the right-hand side will be zero---\eqref{eq:linearintersection} will be empty with probability one.

By the above discussion, we have that
\[
\bbE_{\fkf}\vol_{n-q}\mcZ(\fkf,\bbP^n_\bbR)=\frac{\vol_{n-q}\bbP^{n-q}_{\bbR}}{\vol_{n}\bbP^{n}_{\bbR}}\,\bbE_{\fkf}\bbE_{\fkH_1,\ldots,\fkH_{n-q}}\#(\mcZ(\fkf,\bbP^n_\bbR)\cap \fkH_1\cap\cdots \cap \fkH_{n-q}),
\]
and so, by Tonelli's theorem,
\begin{equation}\label{eq:kinematic}
\bbE_{\fkf}\vol_{n-q}\mcZ(\fkf,\bbP^n_\bbR)=\frac{\vol_{n-q}\bbP^{n-q}_{\bbR}}{\vol_{n}\bbP^{n}_{\bbR}}\,\bbE_{(\fkf,\fkl_1,\ldots,\fkl_{n-q})}\# \mcZ((\fkf,\fkl_1,\ldots,\fkl_{n-q}),\bbP^n_\bbR). 
\end{equation}

Assume temporarily that, for all $i$,
\[
\# \mcZ((\fkf,\fkl_1,\ldots,\fkl_{n-q}),H_i)=0,
\]
where $H_i:=\mcZ(X_i,\bbP^n_\bbR)$ is the $i$th coordinate hyperplane in $\bbP^n_\bbR$, with probability one. (We will prove this at the end.) Then
\[
\mcZ((\fkf,\fkl_1,\ldots,\fkl_{n-q}),\bbP_\bbR^n)= \bigcup\{\mcZ((\fkf,\fkl_1,\ldots,\fkl_{n-q}),S\bbP_+^n)\mid S=\mathrm{diag}(\pm1,\ldots,\pm1)\},
\]
where $\bbP_+^n:=\{x\in \bbP^n\mid x_0>0,\ldots,x_n>0\}$, with probability one; and therefore, by symmetry,
\begin{equation}\label{eq:volumeofpositiveprojective}
   \bbE_\fkf\,\vol_{n-q}\mcZ(\fkf,\bbP^n_\bbR)=2^n\,\frac{\vol_{n-q}\bbP^{n-q}_{\bbR}}{\vol_{n}\bbP^{n}_{\bbR}}\,\bbE_{\fkf,\fkl_1,\ldots,\fkl_{n-q}}\#\mcZ((\fkf,\fkl_1,\ldots,\fkl_{n-q}),\bbP_+^n). 
\end{equation}
Now, by Theorem~\ref{theo:MainTheoremNumberZeros},
\[
\bbE_\fkf\,\vol_{n-q}\mcZ(\fkf,\bbP_+^n)\leq \frac{1}{4^n}\,\frac{\vol_{n-q}\bbP^{n-q}_{\bbR}}{\vol_{n}\bbP^{n}_{\bbR}}\,(n(n+1))^{n-q}\prod_{k=1}^qt_k(t_k-1),
\]
since the $\fkl_k$ have support of size $n+1$, and we obtain the desired upper bound. 

We now prove our claim regarding coordinate hyperplanes. Without loss of generality assume that $i=n$. To show that
\[
\mcZ((\fkf,\fkl_1,\ldots,\fkl_{n-q}),H_n)
\]
is empty, it is enough to show that the random overdetermined system
\[
\fkg:=(\fkf,\fkl_1,\ldots,\fkl_{n-q})(X_0,\ldots,X_{n-1},0),
\]
obtained by setting $X_n$ equal to $0$, has no zeros in $\bbP^{n-1}_{\bbR}$. But this is precisely what Proposition~\ref{prop:withsimplex} states. 
\end{proof}

\begin{proof}[Proof of Proposition~\ref{prop:withsimplex}]
We will prove this by induction on $n$. The statement is obvious for $n=0$. Consider the system
\[\fkg_1(X_0,\ldots,X_{n-1},0),\ldots,\fkg_n(X_0,\ldots,X_{n-1},0),\]
then, by the induction hypothesis, this system does not have any zero in $\bbP^{n-1}_{\bbR}$ with probability one. In other words, as adding equations can only reduce the number of zeros, 
\[
\mcZ(\fkg,\bbP^n_{\bbR})\cap \mcZ(X_n,\bbP^n_{\bbR})=\varnothing
\]
with probability one. Now, the same argument works, if instead of $X_n$ we consider $X_i$. Therefore, for all $i$,
\[\mcZ(\fkg,\bbP^n_{\bbR})\cap \mcZ(X_i,\bbP^n_{\bbR})=\varnothing\]
with probability one. Now, we only need to show that
\[
\bigcup\{\mcZ(\fkg,S\bbP_+^n)\mid S=\mathrm{diag}(\pm1,\ldots,\pm1)\}=\varnothing,
\]
with probability one, where $\bbP_+^n:=\{x\in \bbP^n\mid x_0>0,\ldots,x_n>0\}$. Now, by symmetry, it is enough to show that, with probability one,
\[
\mcZ(\fkg,\bbP_+^n)=\varnothing.
\]
However, the latter follows from the fact that a generic overdetermined system of Laurent polynomials does not have zeros in $(\bbC^*)^n$~\cite[Lemma 1]{yu2016}. This is not explicitly stated in \cite[Lemma 1]{yu2016}, but it can be easily seen by considering the overdetermined system of $q>n$ equations in $n$ variables as a system in $q$ variables after adding $q-n$ dummy variables---then, by \cite[Lemma 1]{yu2016}, the corresponding ideal is not propper for generic polynomials and, by evaluating to 1 the dummy variables, so it was not the original ideal.
\end{proof}

\begin{remark}
A shorter proof of Proposition~\ref{prop:withsimplex} can be obtained by applying directly the results from \cite[Ch. 8]{gkzbook}.
\end{remark}

\subsection*{Acknowledgements}
We thank Alicia Dickenstein for always making time to share her knowledge whenever we asked a question. 
We thank Maurice Rojas for inspiration and his enthusiasm for fewnomials.  We are thankful to Peter Bürgisser and Elias Tsigaridas for helpful discussions. We also thank the reviewers for insightful comments about the paper that helped improving the final version of the paper.

A.E. and J.T.-C. thanks Paul Breiding, Sonja Petrovic and Gregory G. Smith for organizing the \emph{BIRS Workshop ``Random Algebraic Geometry'' (23w5070)} at the Banff International Research Station, in Canada, from April 16 to April 21, 2023, which served as inspiration for this paper. 

A.E. is supported by NSF CCF 2110075 and NSF 2414160.

M.T. is funded by the European Union under the Grant Agreement no. 101044561, POSALG. 
Views and opinions expressed are those of the author(s) only and do not necessarily reflect 
those of the European Union or European Research Council (ERC). Neither the 
European Union nor ERC can be held responsible for them.

J.T.-C. thanks Evgeniya Lagoda for her constant support and Jazz G. Suchen for his insightful suggestion regarding equation~\eqref{eq:intprobformula_phiform}. 
J.T.-C. also thanks Alperen A. Ergür, Eduardo Dueñez, José Iovino, José Morales, Nikos Salingaros, Chris La Valle, Chris Duffer and the UTSA Problem Solving Club (specially Thamara, Rachell, Adam, Taylor, Lila and Alyssa) for making him feel welcome at the University of Texas at San Antonio (UTSA) during his postdoctoral stay there; Jeaheang Bang and Huan Xu for the nice atmosphere at the postdocs' office; and Alejandra Vincencio for making the official paperwork much easier during this stay at UTSA.

{\small 
\printbibliography
}

\appendix
\section{Borel measures and Dynkin's lemma}\label{sec:meeasures}

The point of this appendix is to make clearer, for readers unfamiliar with measure theory, the measure-theoretic arguments underlying the proof of Theorem~\ref{theo:intprobformula}.

Recall the following definitions for families of sets.

\begin{defi}
Let $X$ be a set and $\mcS\subset \mathcal{P}(X)$ a non-empty collection of subsets of $X$. We say that:
\begin{enumerate}
    \item[($\sigma$)] \cite[4.1 Definition]{aliprantisborder2006book} $\mcS$ is a \emph{$\sigma$-algebra} if $\mcS$ contains the empty set and it is closed under complements and countable pairwise disjoint unions, i.e.,
    \begin{itemize}
        \item $\varnothing\in \mcS$.
        \item for all $A\in\mcS$, $X\setminus A\in\mcS$.
        \item for every pairwise disjoint numerable subfamily $\{A_n\}_{n\in\bbN}$ of $\mcS$, $\bigcup_{n\in\bbN}A_n\in \mcS$.
    \end{itemize}
    \item[$(\lambda)$] \cite[4.9 Definition]{aliprantisborder2006book} $\mcS$ is a \emph{$\lambda$-system} if $\mcS$ contains $X$ and it is closed under relative complements and monotone numerable unions, i.e.,
    \begin{itemize}
        \item $X\in \mcS$.
        \item for all $A,B\in\mcS$ such that $B\subseteq A$, $A\setminus B\in\mcS$.
        \item for every numerable subfamily $\{A_n\}_{n\in\bbN}$ of $\mcS$ that is increasing (for all $n$, $A_n\subseteq A_{n+1}$), $\bigcup_{n\in\bbN}A_n\in \mcS$.
    \end{itemize}
    \item[$(\pi)$] \cite[4.9 Definition]{aliprantisborder2006book} $\mcS$ is a \emph{$\pi$-system} if $\mcS$ is closed under finite intersections, i.e., for every $A,B\in\mcS$, $A\cap B\in\mcS$.
\end{enumerate}
\end{defi}

Given any collection of sets, we can consider the $\sigma$-algebra, $\lambda$-system and $\pi$-system that it generates, by considering the smallest $\sigma$-algebra, $\lambda$-system or $\pi$-system that contains it. For a topological space, the following $\sigma$-algebra is essential.

\begin{defi}\cite[4.14 Definition]{aliprantisborder2006book}
Let $X$ be a topological space, its \emph{Borel $\sigma$-algebra}, $\mcB(X)$, is the $\sigma$-algebra generated by the collection of its open subsets.
\end{defi}

The Dynkin's lemma allows us to extend statements from $\pi$-systems to the $\sigma$-algebra they generate by checking that they are satisfied for a $\lambda$-system.

\begin{theo}[Dynkin's lemma]\label{theo:DynkinLemma}\cite[4.11 Dynkin’s Lemma]{aliprantisborder2006book}
Let $X$ be a set and $\mcS$ collection of subsets of $X$. If $\mcS$ is a $\pi$-system, then the $\lambda$-system it generates is a $\sigma$-algebra.    
\end{theo}

Recall that a topological space $X$ is \emph{$\sigma$-compact} if we can write $X$ as a countable monotone union of compact sets and that a \emph{Borel measure} of $X$ is a measure of the form
\[
\mu:\mcB(X)\rightarrow [0,\infty].
\]
The following theorem lies at the core of the proof of Theorem~\ref{theo:intprobformula}.

\begin{theo}\label{theo:measuretheoryargument}
Let $X$ be a $\sigma$-compact topological space and $\mu$ and $\nu$ Borel measures on $X$. Assume that there is a collection of open subsets $\mcU$ such that:
\begin{enumerate}
\item[(U0)] $\mcU$ is a \emph{base} for the topology of $X$, i.e., every open set in $X$ is an union of open sets in $\mcU$.
\item[(U1)] $\mcU$ is closed under taking subsets, i.e., if $U$ and $V$ are open sets, $V\subseteq U$ and $U\in\mcU$, then $V\in \mcU$.
\item[(U2)] for all $U\in \mcU$, $\mu(U)$ and $\nu(U)$ are finite.
\item[(U3)] $\mu$ and $\nu$ agree on $\mcU$, i.e., for all $U\in\mcU$, $\mu(U)=\nu(U)$.
\end{enumerate}
Then $\mu=\nu$.
\end{theo}
\begin{proof}
By assumption, there is an increasing sequence of compact subsets $\{X_n\}_{n\in\bbN}$ such that $X=\bigcup_{n\in\bbN}X_n$. Assume without loss of generality that $X_0=\varnothing$. For each $n\in\bbN$, let
\[
\mu_n:=\mu_{|\mcB(X_n)}~\text{ and }~\nu_n:=\nu_{|\mcB(X_n)}
\]
be the restriction of the Borel-measures $\mu$ and $\nu$ to $X_n$. If for all $n\in\bbN$, $\mu_n=\nu_n$, then $\mu=\nu$. To see this, take $B\in\mcB(X)$ and observe that
\[
\mu(B)=\sum_{k=0}^\infty\mu(B\cap (X_{k+1}\setminus X_k))=\sum_{k=0}^\infty\mu_n(B\cap (X_{k+1}\setminus X_k))
\]
and 
\[
\nu(B)=\sum_{k=0}^\infty\nu(B\cap (X_{k+1}\setminus X_k))=\sum_{k=0}^\infty\nu_n(B\cap (X_{k+1}\setminus X_k)).
\]

Fix an arbitrary $n\in\bbN$. We consider 
\[
\mcU_n:=\{U\cap X_n\mid U\in\mcU\},
\]
which is a collection of open subsets of $X_n$, and
\[
\Lambda_n:=\{B\in \mcB(X_n)\mid \mu_n(B)=\nu_n(B)\}.
\]
If we show that $\mcU_n$ satisfies (U0), (U1), (U3) and (U4) for $\mu_n$ and $\nu_n$, then $\mcU_n$ is a $\pi$-system contained in $\Lambda_n$ whose generated $\sigma$-algebra is $\mcB(X_n)$. If we show, moreover, that $\Lambda_n$ is a $\lambda$-system, then, by Dynkin's lemma (Theorem~\ref{theo:DynkinLemma}), 
$\Lambda_n=\mcB(X_n)$
and we are done, since then $\mu_n=\nu_n$.

We show now that $\mcU_n$ satisfies (U0), (U1), (U3) and (U4) for $\mu_n$ and $\nu_n$:
\begin{itemize}
\item $\mcU_n$ satisfies (U0) and (U1), by the definition of the subspace topology and the construction of $\mcU_n$.
\item For checking (U2) and (U3) for $\mcU_n$, we only need to write for $U\in \mcU$,
\[
\mu(U\cap X_n)=\mu(U)-\mu(U\setminus X_n)=\nu(U)-\nu(U\setminus X_n)=\mu(U\cap X_n).
\]
This is possible, because $U\setminus X_n\in \mcU$, by (U1); $\mu(U)$, $\nu(U)$, $\mu(U\setminus X_n)$ and $\nu(U\setminus X_n)$ are finite, by (U2); and 
$\mu(U)=\nu(U)$ and $\mu(U\setminus X_n)=\nu(U\setminus X_n)$, by (U3). 
\end{itemize}

We show now that $\Lambda_n$ is a $\lambda$-system:
\begin{itemize}
\item Since $X_n$ is compact and $\mcU_n$ is an open cover, we have that there are $U_1,\ldots,U_\ell\in\mcU_n$ such that
\[X_n=U_1\cup\cdots\cup U_\ell.\]
But then, by the inclusion-exclusion principle, (U2) and (U3), $\mu_n(X_n)=\nu(X_n)$. Hence $X_n\in\Lambda_n$.

Moreover, this argument shows that $\mu_n(X_n)=\nu(X_n)$ is finite, and so $\mu_n(A)$ and $\nu_n(A)$ are finite for all $A\in\mcB(X_n)$.
\item If $A,B\in \Lambda_n$ and $B\subseteq A$, then
\[
\mu_n(A\setminus B)=\mu_n(A)-\mu_n(B)=\nu_n(A)-\nu_n(B)=\nu_n(A\setminus B),
\]
where the middle equality follows from $A,B\in\Lambda_n$, since $\mu_n(A)$, $\mu_n(B)$, $\nu_n(A)$ and $\nu_n(B)$ are finite. Hence $A\setminus B\in \Lambda_n$.
\item Let $\{A_k\}_{k\in\bbN}\subset \Lambda_n$ is an increasing family of subsets. Without loss of generality, assume that $A_0=\varnothing$. Then
\[
\mu_n\left(\bigcup_{k\in \bbN}A_k\right)=\sum_{k\in\bbN}\mu_n(A_{k+1}\setminus A_k)
\]
and
\[
\nu_n\left(\bigcup_{k\in \bbN}A_k\right)=\sum_{k\in\bbN}\nu_n(A_{k+1}\setminus A_k).
\]
By the previous paragraph and the assumption $\{A_k\}\subset\Lambda_n$, the right-hand sides are equal. Therefore the left-hand sides are also equal, and thus $\bigcup_{k\in\bbN}A_k\in \Lambda_n$.
\end{itemize}
The proof is complete.
\end{proof}

\end{document}